\documentclass[10.5pt, a4paper]{article}
\usepackage{amsmath}
\usepackage{fancyhdr}
\usepackage[misc]{ifsym}
\usepackage{color}
\usepackage{amsmath}
\usepackage{amsfonts}
\usepackage{dsfont}
\usepackage{mathrsfs}
\usepackage{bbm}
\usepackage{latexsym}
\usepackage{graphicx}
\usepackage{epstopdf}
\usepackage{amssymb}
\usepackage{indentfirst}
\usepackage{subfigure}
\usepackage{booktabs}
\usepackage{authblk}
\usepackage{cite}
\usepackage{float}
\usepackage[numbers,sort&compress]{natbib}
\usepackage{caption}
\usepackage{booktabs}
\usepackage{threeparttable}
\usepackage{multicol}
\usepackage{multirow}
\usepackage{geometry}
\usepackage{lineno}
\usepackage{amsbsy}
\usepackage{bm}
\usepackage{dcolumn}
\usepackage{mathrsfs}
\usepackage[title]{appendix}
\usepackage{stmaryrd}

\newtheorem{remark}{Remark}[section]
\newtheorem{theorem}{Theorem}[section]
\newtheorem{lemma}[theorem]{Lemma}
\newtheorem{corollary}[theorem]{Corollary}

\newtheorem{definition}[theorem]{Definition}

\newtheorem{example}[theorem]{Example}

\newtheorem{assumption}{Assumption}
\numberwithin{equation}{section}

\def\qed{\hfill$\Box$\vspace{8pt}}

\begin{document}
\title{More than one Author with different Affiliations}
\author[1]{Qigang Liang}
\author[2]{Wei Wang}
\author[1,3]{Xuejun Xu}

\affil[1]{\small School of Mathematical Science, Tongji University, Shanghai 200092, China, qigang$\_$liang@tongji.edu.cn}
\affil[2]{\small School of Mathematics, University of Minnesota, Minneapolis, MN 55455, USA, wang9585@umn.edu}
\affil[3]{\small Institute of Computational Mathematics, AMSS, Chinese Academy of Sciences, Beijing 100190, China, xxj@lsec.cc.ac.cn}
\title{A Two-Level Block Preconditioned Jacobi-Davidson Method for Multiple and Clustered Eigenvalues of Elliptic Operators}\date{}
\maketitle

{\bf{Abstract}:}\ \ In this paper, we propose a two-level block preconditioned Jacobi-Davidson (BPJD) method for efficiently solving discrete eigenvalue problems resulting from finite element approximations of $2m$th ($m = 1, 2$) order symmetric elliptic eigenvalue problems. Our method works effectively to compute the first several eigenpairs, including both multiple and clustered eigenvalues with corresponding eigenfunctions, particularly. The method is highly parallelizable by constructing a new and efficient preconditioner using an overlapping domain decomposition (DD). It only requires computing a couple of small scale parallel subproblems and a quite small scale eigenvalue problem per iteration. Our theoretical analysis reveals that the convergence rate of the method is bounded by $c(H)(1-C\frac{\delta^{2m-1}}{H^{2m-1}})^{2}$, where $H$ is the diameter of subdomains and $\delta$ is the overlapping size among subdomains. The constant $C$ is independent of the mesh size $h$ and the internal gaps among the target eigenvalues, demonstrating that our method is optimal and cluster robust. Meanwhile, the $H$-dependent constant $c(H)$ decreases monotonically to $1$, as $H \to 0$, which means that more subdomains lead to the better convergence rate. Numerical results supporting our theory are given.

{\bf{Keywords}:}\ \ PDE eigenvalue problems, finite element discretization, multiple and clustered eigenvalues, preconditioned Jacobi-Davidson method, overlapping domain decomposition. \hspace*{2pt}

\section{Introduction}

\par Solving large scale eigenvalue problems arising from the discretization of partial differential operators by finite element methods is one of the fundamental problems in modern science and engineering. The problem is essential and has been extensively studied in the literature (see, e.g., \cite{Boffi,Lin1,Auke1,yang2015shifted,xu2020parallel,Dominik1,Cai,WX1,WX2}). However, unlike boundary value problems, there are fewer parallel solvers available for solving PDE eigenvalue problems, especially when it comes to computing multiple and clustered eigenvalues, which poses a greater challenge. To address this issue, we propose a two-level block preconditioned Jacobi-Davidson (BPJD) method that can compute multiple and clustered eigenvalues. Our method utilizes a parallel preconditioner constructed through an overlapping domain decomposition (DD), with a rigorous theoretical analysis, which demonstrates to be optimal and scalable. Specifically, the convergence rate does not deteriorate as the fine mesh size $h\to 0$, or the number of subdomains increases. In particular, our method is cluster robust, meaning that the convergence rate is not negatively impacted by gaps among the clustered eigenvalues.

\par For elliptic eigenvalue problems, Babuška and Osborn in \cite{babuvska1989finite} employed the finite element method to compute eigenpairs. Two-grid methods have also been widely adopted, and achieve asymptotic optimal accuracy under the conditions that $h = O(H^{i})$ respectively (here $H$ represents the coarse mesh size, $i = 2$ or $4$), as evidenced in \cite{Aihui,Xiaoliang,Yang1}. For discrete PDE eigenvalue problems, various classical iterative algorithms have been applied (see \cite{Parlett1,YousefSaadbook1,SIAMReviewnew1}), 
among which the Jacobi-Davidson method proposed in \cite{sleijpen} is one of the most popular methods in practice. The Jacobi-Davidson method has been successfully applied to a variety of practical computations, including Maxwell eigenvalue problems \cite{Lin1}, magnetohydrodynamics (MHD) eigenvalue problems \cite{Auke1}, polynomial PDE eigenvalue problems \cite{Dominik1,Cai}, and computations of large singular value decomposition \cite{Huang1} and so forth.

\par When dealing with large scale discrete PDE eigenvalue problems, preconditioning techniques are usually required (see \cite{Knyazeveigensolvers1}). Cai et al. \cite{CaiZhiqiang1997} and McCormick \cite{McCormick1} have proposed several multigrid methods for computing the eigenpairs. Recently, a range of multilevel correction methods have been studied for solving elliptic eigenvalue problems (see \cite{chen2016full,xu2020parallel}). Yang et al. \cite{yang2015shifted} has proposed an alternative multilevel correction method based on the shift and inverse technique for solving elliptic eigenvalue problems.

\par It is widely known that domain decomposition methods perform better than multigrid (MG) methods in terms of parallelism. Lui \cite{lui2000domain} proposed some two-subdomain DD methods to compute the principal eigenpair through solving an interface problem. For many subdomain cases, Maliassov \cite{maliassov1996schwarz} constructed a Schwarz alternating method to solve the eigenvalue problem, which can be proven to be convergent under a suitable assumption. According to dealing with an interface condition, Genseberger \cite{GensebergerMenno1} presented some eigensolvers by combining the Jacobi-Davidson method with non-overlapping domain decomposition methods. Zhao et al. \cite{Cai,zhao2016parallel} proposed a two-level preconditioned Jacobi-Davidson (PJD) method for a quintic polynomial eigenvalue problem. Wang and Xu \cite{WX2} developed a domain decomposition method to precondition the Jacobi-Davidson correction equation in one step during every outer iteration, with theoretical analysis for $2m$th ($m = 1, 2$) order elliptic operators presented. Wang and Zhang \cite{MR4396363} designed DD methods for eigenvalue problems based on the spectral element discretization. More recently, Liang and Xu \cite{LQGXXJ} presented a two-level preconditioned Helmholtz-Jacobi-Davidson (PHJD) method for the Maxwell eigenvalue problem, which works well in practical computations and has been proven to be optimal and scalable. 

\par For computing multiple and clustered eigenvalues of PDE eigenvalue problems, Knyazev and Osborn in \cite{knyazev2006new} gave the a priori error estimates for multiple and clustered eigenvalues of symmetric elliptic eigenvalue problems. Dai et al. \cite{MR3407250} developed an a posteriori error estimator for multiple eigenvalue, and proved the convergence and quasi-optimal complexity of the adaptive finite element methods (AFEM). Subsequently, Gallistl \cite{MR3347459} studied clustered eigenvalues, and proved the convergence and quasi-optimal complexity of AFEM. This idea has been further extended to the higher-order AFEM \cite{MR3532806}, the non-conforming AFEM \cite{MR3259027}, and the mixed AFEM \cite{MR3647956}. Lately, Canc\`es et al. \cite{MR4136540} presented an a posteriori error estimator for conforming finite element approximations of multiple and clustered eigenvalues of symmetric elliptic operators. They introduce some novel techniques to estimate the error in the sum of the eigenvalues. Additionally, for large scale discrete PDE eigenvalue problems, designing efficient solvers for multiple and clustered eigenvalues is a significant task, and will exceedingly benefit the high-performance computation. However, the theoretical analysis of the two-level PJD method in \cite{WX1,WX2} or the two-level PHJD method in \cite{LQGXXJ} is limited to the simple principal eigenvalue. To this end, we aim to construct efficient solvers with rigorous analysis for the first several eigenvalues comprehensively in this paper, including multiple and clustered eigenvalues, as well as their corresponding eigenspaces.

\par It is important to highlight that analyzing the two-level BPJD method from simple to multiple and clustered eigenvalues is a challenging task.
Firstly, it is essential to ensure that all constants involved in the convergence rate constants are independent of the internal gaps of the target eigenvalues. Secondly, as the eigenspace dimension corresponding to multiple eigenvalues is greater than one, most techniques developed for simple eigenvalues in \cite{WX1,WX2,LQGXXJ} are not applicable. Thirdly, measuring the distance requires the use of Hausdorff distance or gap in the theoretical analysis, rather than vector norm, which poses additional challenges. In this paper, using a combination of techniques, including constructing a well-designed auxiliary eigenvalue problem, developing a stable decomposition for the error space of target eigenvalues, and providing technical estimates for the gap between closed subspaces in Hilbert space, we successfully overcome the aforementioned difficulties. Consequently, we demonstrate that our method achieves optimal, scalable, and cluster robust convergence result, i.e.,
\begin{equation}\notag
\sum_{i=1}^{s}(\lambda_{i}^{k+1}-\lambda_{i}^{h})\leq \gamma \sum_{i=1}^{s}(\lambda_{i}^{k}-\lambda_{i}^{h}),
\end{equation}
where $\lambda_{i}^{k}$ is the current iterative approximation of the $i$th discrete eigenvalue $\lambda_{i}^{h}$, $\gamma=c(H)(1-C\frac{\delta^{2m-1}}{H^{2m-1}})^{2}$, the constant $C$ is independent of $h$, $H$, $\delta$ and internal gaps among the first $s$ eigenvalues, and the $H$-dependent constant $c(H)$ decreases monotonically to $1$, as $H\to 0$. Moreover, we have not any assumption on the relationship between $H$ and $h$, as well as the internal gaps among the first $s$ eigenvalues. Numerical results presented in this paper verify our theoretical findings.

\par The rest of this paper is organized as follows: Some preliminaries are introduced in Section 2. In Section 3, the two-level BPJD method for $2m$th order symmetric elliptic eigenvalue problems is proposed. Some properties of subspace method are presented in Section 4 and the main convergence analysis is given in Section 5. Finally we present our numerical results in Section 6 and the conclusion in Section 7.

\section{Model problems and preliminaries}

\par In this section, we first introduce some notations and model problems in subsection 2.1, the corresponding discrete counterpart in subsection 2.2, and then present some results on domain decomposition methods in subsection 2.3.

\subsection{Model problems}

\par Throughout this paper, we use the standard notations for the Sobolev spaces $W^{m,q}(\Omega)$ and $W_{0}^{m,q}(\Omega)$ with their associated norms and seminorms. We denote by $H^{m}(\Omega):=W^{m,q}(\Omega)$, $H_{0}^{m}(\Omega):=W_{0}^{m,q}(\Omega)$ for $q=2$, and denote by $L^{2}(\Omega):=H^{m}(\Omega)$ for $m=0$. Consider the Laplacian and biharmonic eigenvalue problems as follows:
\begin{equation}\label{Laplacian}
     \begin{cases}
         -\triangle u=\lambda u\ \ &\text{in $\Omega,$}\\
         \ \ \ \ \ u=0\ \ &\text{on $\partial\Omega,$}
     \end{cases}
\end{equation}
and
\begin{equation}\label{Biharmonic}
     \begin{cases}
         \triangle^{2} u=\lambda u\ \ &\text{in $\Omega,$}\\
         \frac{\partial u}{\partial \bm{n}}=u=0\ \ &\text{on $\partial\Omega.$}
     \end{cases}
\end{equation}
For simplicity, we assume that $\Omega$ is a convex polygonal domain in $\mathcal{R}^{2}$ and $\partial\Omega$ is the boundary of $\Omega$. We denote by $\bm{n}$ the unit outward normal vector of $\partial \Omega$.
\par The variational form of $2m$th order symmetric elliptic eigenvalue problems may be written as:
 \begin{equation}\label{Variational}
     \begin{cases}
       \text{Find $(\lambda,u)\in \mathcal{R}\times V$ such that $b(u,u)=1,$}\\
       a(u,v)=\lambda b(u,v)\ \ \ \ \forall\ v\in V,
     \end{cases}
\end{equation}
 where $V:=H_{0}^{m}(\Omega)$, the bilinear forms $a(\cdot,\cdot): V\times V\to \mathcal{R}$, $b(\cdot,\cdot): L^{2}(\Omega)\times L^{2}(\Omega)\to \mathcal{R}$ are symmetric and positive. Define $b(u,v):=\int_{\Omega}uv dx$ for all $u,v\in L^{2}(\Omega)$ and $||v||^{2}_{b}:=b(v,v)$ for all $v\in L^{2}(\Omega)$.
Specifically, for \eqref{Laplacian}, $$ a(u,v)=\int_{\Omega}\nabla{u}\cdot \nabla{v} dx$$ for all $u,v\in V=H_{0}^{1}(\Omega),$ and for \eqref{Biharmonic}, $$ a(u,v)=\int_{\Omega}\triangle{u}\triangle{v} dx$$ for all $u,v\in V=H_{0}^{2}(\Omega).$
It is easy to see that $a(\cdot,\cdot)$ constructs an inner product on $V$ and we define $||v||^{2}_{a}:=a(v,v)$ for all $v\in V$. For convenience, we denote by $Rq(v):=\frac{a(v,v)}{b(v,v)}>0$ for all $v\ (\ne 0)\in V$ the Rayleigh quotient functional. We also define $Rt(v):=\frac{1}{Rq(v)}$ for all $v\ (\ne 0)\in V.$
\par Define a linear operator $T: L^{2}(\Omega)\to V$ such that for any $f\in L^{2}(\Omega),$
\begin{equation}\label{equation_DefineT}
a(Tf,v)=b(f,v)\ \ \ \ \forall\ v\in V.
\end{equation}
Since $a(\cdot,\cdot)$ and $b(\cdot,\cdot)$ are symmetric and $V$ is embedded compactly in $L^{2}(\Omega)$, we know that $T:L^{2}(\Omega)\to L^{2}(\Omega)$ is compact and symmetric. Moreover, $T:V\to V$ is also compact and symmetric. By the Hilbert-Schmidt Theorem, we get that $Tu_{i}=\mu_{i}u_{i}\ (\mu_{i}=(\lambda_{i})^{-1})$, and the eigenvalues of \eqref{Variational} are
$$\lambda_{1}\leq \lambda_{2}\leq,...,\leq\lambda_{n}\to +\infty,$$
and the corresponding eigenvectors are $u_{1},u_{2},...,u_{n},...,$
which satisfy $a(u_{i},u_{j})=\lambda_{i}b(u_{i},u_{j})=\lambda_{i}\delta_{ij}$ ($\delta_{ij}$ represents the Kronecker delta). In the sequence $\{\lambda_{i}\}_{i=1}^{+\infty}$, $\lambda_{i}$ is repeated according to its geometric multiplicity. For convenience, we call $d^{-}_{i}\ (:=\lambda_{i}-\lambda_{i-1})$ and $d^{+}_{i}\ (:=\lambda_{i+1}-\lambda_{i})$ as the left gap and the right gap of the eigenvalue $\lambda_{i}$, respectively. In particular, $d^{-}_{1}:=\lambda_{1}$.
\par We are interested in the first $s$ eigenvalues $\{\lambda_{i}\}_{i=1}^{s}$ and the corresponding eigenvectors $\{u_{i}\}_{i=1}^{s}$. For our theoretical analysis, we first introduce a reasonable assumption.
\begin{assumption}\label{Assumption1}
Assume that there is an `obvious' gap between the $s^{th}$ eigenvalue $\lambda_{s}$ and the $(s+1)^{th}$ eigenvalue $\lambda_{s+1}$.
\end{assumption}
\begin{remark}
Assumption $\ref{Assumption1}$ excludes two cases:\ (i) $\lambda_{s}=\lambda_{s+1}$, (ii) $\lambda_{s}\approx\lambda_{s+1}$,
but there are not any assumptions about the left gaps of the eigenvalues $\{\lambda_{i}\}_{i=2}^{s}$. In practical computation, either for case $(i)$ or for case $(ii)$, we may consider the first $s+s_{1}$ eigenvalues so that $\lambda_{s+s_{1}}$ and $\lambda_{s+s_{1}+1}$ satisfy the Assumption 1, where $s_{1}(\geq 1)$ is a positive integer. So actually this assumption is not a limitation for our practical computation.
\end{remark}
\par It is known that the following spacial decomposition property holds
\begin{equation}\label{Equation_ConstinuousDecomposition}
V=U_{s}\oplus U_{s+1},
\end{equation}
where $U_{s}=$span$\{u_{1},u_{2},...,u_{s}\}$, $\oplus$ denotes the orthogonal direct sum with respect to $b(\cdot,\cdot)$ (also $a(\cdot,\cdot)$) and $U_{s+1}$ is the orthogonal complement of $U_{s}$.
\par In order to measure the `distance' between two closed subspaces included in a Hilbert space, we introduce the following definition. For more details, please see \cite{Kato}, Section 2 in \cite{knyazev2006new} and references therein.
\begin{definition}\label{Definition_gap}
For any Hilbert space $(X,(\cdot,\cdot))$, define $\Sigma_{X}:=\{\ W\ |\ \text{ $W$ is a closed subspace of $X$}\}.$
A binary mapping $\theta$ (called as the gap) $:\Sigma_{X}\times\Sigma_{X} \to [0,1]$ is defined by
\begin{equation}\notag
\theta(W_{1},W_{2})=\max\{\sin\{W_{1};W_{2}\},\sin\{W_{2};W_{1}\}\}\ \ \ \text{for all}\ \ W_{1},W_{2}\in \Sigma_{X},
\end{equation}
where
$$\sin\{W_{1};W_{2}\}=\sup_{u\in W_{1},||u||=1}\inf_{v\in W_{2}}||u-v|| \ \ \ \text{for all}\ \ W_{1},W_{2}\in \Sigma_{X},$$
with $||\cdot||$ being a norm induced by $(\cdot,\cdot)$ defined on $X$. If $W_{1}=0$, set $\sin(W_{1},W_{2})=0\ \text{for all}\ \ W_{2}\in \Sigma_{X}.$ If $W_{2}=0$, set $\sin(W_{1},W_{2})=1\ \text{for all}\ \ W_{1}(\ne 0)\in \Sigma_{X}.$
\end{definition}

\begin{remark}\label{Remark_Hausdorff}
For any Hilbert space $(X,(\cdot,\cdot))$, if $W_{1},W_{2}\in \Sigma_{X}$ and $\dim(W_{1})=\dim(W_{2})<+\infty$, it is easy to know that
$$\theta(W_{1},W_{2})=\sin\{W_{1};W_{2}\}=\sin\{W_{2};W_{1}\}.$$
For any $W_{1},W_{2},W_{3}\in\Sigma_{X}$ and $\dim(W_{1})=\dim(W_{2})=\dim(W_{3})<+\infty$, it is easy to check that $$\sin\{W_{1};W_{2}\}\leq \sin\{W_{1};W_{3}\}+\sin\{W_{3};W_{2}\}.$$
If $W_{1}={\rm span}\{u\}$, then $\sin\{W_{1};W_{2}\}$ is denoted through $\sin\{u;W_{2}\}$. Similarly, if $W_{2}={\rm span}\{v\}$, then $\sin\{W_{1};W_{2}\}$ is denoted through $\sin\{W_{1};v\}$.
\end{remark}
\par In the rest of this paper, we shall use the notations $\sin_{b}\{\cdot;\cdot\}$ and $\sin_{a}\{\cdot,\cdot\}$  with respect to  $b(\cdot,\cdot)$ and $a(\cdot,\cdot)$, respectively. We also denote by $\theta_{b}(\cdot,\cdot)$ and $\theta_{a}(\cdot,\cdot)$ the gaps with respect to  $b(\cdot,\cdot)$ and $a(\cdot,\cdot)$, respectively, denote by $\interleave\cdot\interleave_{b}$ and $\interleave\cdot\interleave_{a}$ the operator's norms with respect to  $b(\cdot,\cdot)$ and $a(\cdot,\cdot)$, respectively.

\subsection{Finite element discretization}
\par Let $V^{h}$ be a conforming finite element space based on a shape regular and quasi-uniform triangular or rectangular partition $\mathcal{J}_{h}$ with the mesh size $h$. We consider the discrete variational form of \eqref{Variational} as:
 \begin{equation}\label{DiscreteVariational}
     \begin{cases}
       \text{Find $(\lambda^{h},u^{h})\in \mathcal{R}\times V^{h}$ such that $||u^{h}||_{b}=1,$}\\
       a(u^{h},v)=\lambda^{h} b(u^{h},v)\ \ \ \forall\ v\in V^{h}.
     \end{cases}
\end{equation}
Define a discrete linear operator $T^{h}:L^{2}(\Omega)\to V^{h}$ such that for any $f\in L^{2}(\Omega)$,
\begin{equation}\label{equation_DefineTh}
a(T^{h}f,v)=b(f,v)\ \ \ \forall\ v\in V^{h}.
\end{equation}
It is easy to see that the operator $T^{h}$ is compact and symmetric (For convenience of notations, $T^{h}|_{V^{h}}$ is also denoted through $T^{h}$ in the following). Hence, we get that $T^{h}u_{i}^{h}=\mu_{i}^{h}u_{i}^{h}$ ($\mu_{i}^{h}=(\lambda_{i}^{h})^{-1}$). Meanwhile, the eigenvalues of \eqref{DiscreteVariational} are $\lambda_{1}^{h}\leq \lambda_{2}^{h}\leq,...,\leq\lambda_{nd}^{h}$
and the corresponding eigenvectors are $u_{1}^{h},u_{2}^{h},...,u_{nd}^{h},$
which satisfy $a(u_{i}^{h},u_{j}^{h})=\lambda_{i}^{h}b(u_{i}^{h},u_{j}^{h})=\lambda_{i}^{h}\delta_{ij}$ and $nd=\dim(V^{h})$.
We also define $A^{h}:V^{h}\to V^{h}$ such that
$b(A^{h}u,v)=a(u,v)$ for all $u,v\in V^{h},$ and it is obvious to see that $A^{h}u_{i}^{h}=\lambda_{i}^{h}u_{i}^{h}$.
\par  The finite element space $V^{h}$ may be decomposed as:
\begin{equation}\label{VhDecomposition}
V^{h}=U_{s}^{h}\oplus U^{h}_{s+1}=V_{1}^{h}\oplus V_{2}^{h}\oplus...\oplus V_{s}^{h}\oplus U_{s+1}^{h},
\end{equation}
where $U_{s}^{h}=V_{1}^{h}\oplus V_{2}^{h}\oplus ...\oplus V_{s}^{h}$, $V_{i}^{h}=$ span$\{u_{i}^{h}\}$, $i=1,2,...,s$ and $U_{s+1}^{h}$ denotes the $b(\cdot,\cdot)$-orthogonal (also $a(\cdot,\cdot)$-orthogonal) complement of $U_{s}^{h}$. Let $Q_{s}^{h},Q_{s+1}^{h}$ and $Q_{i,s}^{h}$ ($i=1,2,...,s$) be the $b(\cdot,\cdot)$-orthogonal (also $a(\cdot,\cdot)$-orthogonal) projectors from $V^{h}$ onto $U_{s}^{h},\ U_{s+1}^{h}$ and $V_{i}^{h}$ ($i=1,2,...,s$), respectively.
For any subspace $U\subset V^{h}$, $U^{\perp}$ represents the orthogonal complement of $U$ with respect to $b(\cdot,\cdot)$, and let $Q_{U}$ and $P_{U}$ be the $b(\cdot,\cdot)$-orthogonal and the $a(\cdot,\cdot)$-orthogonal projectors from $V^{h}$ onto $U$, respectively. If $U=$span$\{u\}$, then we denote $Q_{u}=Q_{U}$ and $P_{u}=P_{U}$. Unless otherwise stated, the letters $C$ (with or without subscripts) in this paper denote generic positive constants independent of $h$, $H,\ \delta$ and the left gaps of the eigenvalues $\{\lambda_{i}\}_{i=2}^{s}$, which may be different at different occurrences.

\subsubsection{The Laplacian eigenvalue problem}
\par In order to make the ideas clearer, we use $V^{h}$, the continuous piecewise and linear finite element space with vanishing trace, to approximate the Sobolev space $H_{0}^{1}(\Omega)$ for the Laplacian eigenvalue problem. The following a priori error estimates are useful in this paper. For the first conclusion in Theorem \ref{TheoremPriorLaplacian}, please see \cite{babuvska1989finite} and \cite{knyazev2006new} for more details. For the proof of \eqref{Equation_eigenvalueApriori}, please see Theorem 3.1 and Theorem 3.3 in \cite{knyazev2006new}. In order to focus on more our algorithm in Section 3 and the corresponding theoretical analysis in Section 4 and 5, we give proofs of \eqref{Equation_eigenvectorApriori1} and \eqref{Equation_eigenvectorApriori2} in Appendix.
\begin{theorem}\label{TheoremPriorLaplacian}
Let $\Omega$ be a bounded convex polygonal domain. If Assumption \ref{Assumption1} holds, then the eigenvalues of discrete problem \eqref{DiscreteVariational} $\lambda_{1}^{h},\lambda_{2}^{h},...,\lambda_{s}^{h}$ converge to the eigenvalues of problem \eqref{Variational} $\lambda_{1},\lambda_{2},...,\lambda_{s}$, respectively, as $h\to 0$. Moreover, there exists $h_{0}>0$ such that for $0<h<h_{0}$, the following inequalities hold:
\begin{equation}\label{Equation_eigenvalueApriori}
0\leq\lambda_{i}^{h}-\lambda_{i}\leq Ch^{2},\ \ \ \ i=1,2,...,s,
\end{equation}
and
\begin{equation}\label{Equation_eigenvectorApriori1}
\theta_{a}(U_{s},U_{s}^{h})\leq Ch,
\end{equation}
\begin{equation}\label{Equation_eigenvectorApriori2}
\theta_{b}(U_{s},U_{s}^{h})\leq Ch^{2},
\end{equation}
where the constant $C$ is independent of the left gaps of the eigenvalues $\{\lambda_{i}\}_{i=2}^{s}$, but depends on $d_{s}^{+}$,  $\theta_{a}(U_{s},U_{s}^{h})$ and $\theta_{b}(U_{s},U_{s}^{h})$ denote the gaps between $U_{s}$ and $U_{s}^{h}$ with respect to $||\cdot||_{a}$ and $||\cdot||_{b}$, respectively.
\end{theorem}

\subsubsection{The biharmonic eigenvalue problem}
\par For problem \eqref{Biharmonic}, we shall use $V^{h}$, the Bogner-Fox-Schmit (BFS) finite element space with vanishing trace and vanishing trace of outer normal derivative, to approximate the Sobolev space $H_{0}^{2}(\Omega)$. For more details about BFS finite element, please see \cite{WX2} and references therein. Under the regularity assumption that the eigenfunction $u_{i}\in H^{3}(\Omega)\cap H_{0}^{2}(\Omega)\ (i=1,2,...,s)$ for the biharmonic eigenvalue problem, we also have the same theoretical results as Theorem \ref{TheoremPriorLaplacian}.

\subsection{Domain decomposition}
\par In this subsection, we introduce some results on overlapping domain decomposition.
\par Let $\{\Omega_{l}\}_{l=1}^{N}$ be a coarse shape regular and quasi-uniform partition of $\Omega$, and we denote it by $\mathcal{J}_{H}$. We define $H:=\max\{H_{l}\ |\ l=1,2,...,N \}$, where $H_{l}=\text{diam}(\Omega_{l})$. The fine shape regular and quasi-uniform partition $\mathcal{J}_{h}$ is obtained by subdividing $\mathcal{J}_{H}$. We may construct the finite element spaces $V^{H}\subset V^{h}$ on $\mathcal{J}_{H}$ and $\mathcal{J}_{h}$, respectively. To get the overlapping subdomains $(\Omega_{l}^{'},\ 1\leq l\leq N)$, we enlarge the subdomains $\Omega_{l}$ by adding fine elements inside $\Omega$ layer by layer such that $\partial \Omega_{l}^{'}$ does not cut through any fine element. To measure the overlapping width between neighboring subdomains, we define $\delta:=\min\{\delta_{l}\ |\ l=1,2,...,N\}$, where $\delta_{l}=\text{dist}(\partial\Omega_{l}\setminus\partial\Omega,\partial\Omega_{l}^{'}\setminus\partial\Omega)$. We also assume that $H_{l}$ is the diameter of $\Omega_{l}^{'}$. Let $\Omega_{l,\delta_{l}}\ (\subset \Omega_{l}^{'})$ be the set of the points that are within a distance $\delta_{l}$ of $\partial\Omega_{l}^{'}\setminus\partial\Omega,\ l=1,2,...,N.$ The local subspaces may be defined by $V^{(l)}:=V^{h}\cap H_{0}^{1}(\Omega_{l}^{'})$ (for the Laplacian operator) or $V^{(l)}:=V^{h}\cap H_{0}^{2}(\Omega_{l}^{'})$ (for the biharmonic operator).
It is obvious to see $V^{(l)}\subset V^{h}$ by a trivial extension.

\begin{assumption}\label{Assumption2}
The partition $\{\Omega_{l}^{'}\}_{l=1}^{N}$ may be colored using at most $N_{0}$ colors, in such a way that subdomains with the same color are disjoint. The integer $N_{0}$ is independent of $N$.
\end{assumption}

\par There exists a family of continuous piecewise and linear functions $\{\theta_{l}\}_{l=1}^{N}$ which satisfy the following properties (see \cite{sarkis2002partition} or the Chapter 3 in \cite{ToselliM}):
\begin{equation}\label{unitypartition}
\text{supp}(\theta_{l})\subset \overline{\Omega_{l}^{'}},\ \ \ \
0 \leq \theta_{l}\leq 1,\ \ \ \
\sum_{l=1}^{N}\theta_{l}(x)=1,\ \ x\in \Omega,\ \ \ \
||\nabla{\theta_{l}}||_{0,\infty,\Omega_{l,\delta_{l}}}\leq \frac{C}{\delta_{l}}.
\end{equation}
We also note that $\nabla{\theta_{l}}$ differs from zero only in a strip $\Omega_{l,\delta_{l}}$. The strengthened Cauchy-Schwarz inequality holds over the local subspaces $V^{(l)}$, i.e., there exists $\eta_{pl}\ (0\leq\eta_{pl}\leq 1,\ 1\leq p,l\leq N)$ such that
$$|b(v^{(p)},v^{(l)})|\leq \eta_{pl}||v^{(p)}||_{b}||v^{(l)}||_{b}\ \ \ \ \  \ \forall\ v^{(p)}\in V^{(p)},\ v^{(l)}\in V^{(l)},\ 1\leq p,l\leq N.$$
Let $\rho(\Lambda)$ be the spectral radius of the matrix $(\eta_{pl})_{1\leq p,l\leq N}$, then the following result holds (see \cite{ToselliM}).
\begin{lemma}\label{Lemma_strengthenedCauchySchwarzinequality}
If Assumption \ref{Assumption2} holds, then $\rho(\Lambda)\leq N_{0}.$
Moreover, for any $v^{(p)}\in V^{(p)}, v^{(l)}\in V^{(l)}\ (p,l=1,2,...,N)$,
\begin{equation}\notag
\sum_{p,l=1}^{N}b(v^{(p)},v^{(l)})\leq N_{0}\sum_{l=1}^{N}b(v^{(l)},v^{(l)}),\ \ \ \sum_{p,l=1}^{N}a(v^{(p)},v^{(l)})\leq N_{0}\sum_{l=1}^{N}a(v^{(l)},v^{(l)}).
\end{equation}
\end{lemma}
\par The following result holds in $H^{1}(\Omega_{l}^{'})$ (see \cite{ToselliM}).
\begin{lemma}\label{Lemma_SmallOvelapping}
It holds that
\begin{equation}\notag
||u||_{L^{2}(\Omega_{l,\delta_{l}})}^{2}\leq C\delta_{l}^{2}\{(1+\frac{H_{l}}{\delta_{l}})|u|_{H^{1}(\Omega_{l}^{'})}^{2}+\frac{1}{H_{l}\delta_{l}}||u||^{2}_{L^{2}(\Omega_{l}^{'})}\}\ \ \ \forall\ u\in H^{1}(\Omega_{l}^{'}),\ \ l=1,2,...,N.
\end{equation}
\end{lemma}

\section{The two-level BPJD method}
\par In this section, we present our two-level BPJD method and some remarks about our algorithm.
\par In order to present our new preconditioner, we denote by $Q^{H}:V^{h}\to V^{H}$, $Q^{(l)}:V^{h}\to V^{(l)}\ (l=1,2,...,N)$ $b(\cdot,\cdot)$-orthogonal projectors. We also define $A^{(l)}: V^{(l)}\to V^{(l)}$ such that $b(A^{(l)}v,w)=a(v,w)$ for all $v,w\in V^{(l)},$ and $A^{H}:V^{H}\to V^{H}$ such that $b(A^{H}v,w)=a(v,w)$ for all $v,w\in V^{H}$. For convenience, denote by $B_{0,i}^{k}:=A^{H}-\lambda_{i}^{k}$ and $B_{l,i}^{k}:=A^{(l)}-\lambda_{i}^{k}\ (i=1,2,...,s,\ l=1,2,...,N)$,
where $\lambda_{i}^{k}$ represents the $k$th iterative approximation of the $i$th discrete eigenvalue $\lambda_{i}^{h}$ in Algorithm 3.1. By using a scaling argument, it is easy to check that
\begin{equation}\label{EigenvalueLocal}
\lambda_{min}(B_{l,i}^{k})=O(H^{-2m}),\ \ \ \lambda_{max}(B_{l,i}^{k})=O(h^{-2m}),\ \ m=1,2,\ i=1,2,...,s,\ l=1,2,...,N.
\end{equation}

Corresponding to \eqref{VhDecomposition}, there is a spectral decomposition on the coarse space $V^{H}$ ($s<\dim(V^{H})$):
\begin{equation}\notag
V^{H}=U_{s}^{H}\oplus U_{s+1}^{H}=V_{1}^{H}\oplus V_{2}^{H}\oplus ... \oplus V_{s}^{H}\oplus U_{s+1}^{H},
\end{equation}
where $V_{i}^{H}={\rm span} \{u_{i}^{H}\}$, $u_{i}^{H}$ is the $i$th discrete eigenvector of $A^{H}$, $\ U_{s}^{H}=V_{1}^{H}\oplus V_{2}^{H}\oplus ... \oplus V_{s}^{H}$, $\oplus$ denotes the orthogonal direct sum with respect to $b(\cdot,\cdot)$ (also $a(\cdot,\cdot)$), and $U_{s+1}^{H}$ denotes the orthogonal complement of $U_{s}^{H}$.
Furthermore,
\begin{equation}\label{EigenvalueCoarse}
\lambda_{min}(B_{0,i}^{k}|_{U_{s+1}^{H}})= \lambda_{s+1}^{H}-\lambda_{i}^{k},\ \ \ \lambda_{max}(B_{0,i}^{k})=O(H^{-2m}),\ \ m=1,2, \ i=1,2,...,s.
\end{equation}
We also denote by $Q_{s}^{H}: V^{H}\to U_{s}^{H},\ Q_{s+1}^{H}:V^{H}\to U_{s+1}^{H}$ and $Q_{i,s}^{H}: V^{H}\to V_{i}^{H}$ $b(\cdot,\cdot)$-orthogonal (also $a(\cdot,\cdot)$-orthogonal ) projectors. The core of our two-level BPJD method is to design  parallel preconditioners defined as
\begin{equation}\label{Preconditioner}
(B_{i}^{k})^{-1}=(B_{0,i}^{k})^{-1}Q_{s+1}^{H}Q^{H}+\sum_{l=1}^{N}(B_{l,i}^{k})^{-1}Q^{(l)},
\end{equation}
to solve the block-version Jacobi-Davidson correction equations:
\begin{equation}\label{Equation_JacobiDavidsonCorrtion}
\begin{cases}
\text{Find $t_{i}^{k+1}\in (U^{k})^{\perp},\ i=1,2,...,s,$  such that}\\
b((A^{h}-\lambda_{i}^{k})t_{i}^{k+1},v)=b(r_{i}^{k},v),\ \ \ \ v\in (U^{k})^{\perp},
\end{cases}
\end{equation}
where $U^{k}={\rm span}\{u_{i}^{k}\}_{i=1}^{s}$, $u_{i}^{k}$ is the iterative approximation of $u_{i}^{h}$, and $r_{i}^{k}=\lambda_{i}^{k}u_{i}^{k}-A^{h}u_{i}^{k},\ i=1,2,...,s$.
\begin{table}[H]
\centering
\begin{tabular}{p{15cm}}
\hline
\hline
\textbf{Algorithm 3.1} Two-Level BPJD Algorithm \\
\hline
$\bm{{\rm Step}\ 1}$ Solve the following coarse eigenvalue problem:
\[A^{\widetilde{H}}u_{i}^{\widetilde{H}}=\lambda_{i}^{\widetilde{H}}u_{i}^{\widetilde{H}},\ \ \ b(u_{i}^{\widetilde{H}},u_{j}^{\widetilde{H}})=\delta_{ij},\ \ \ \ i,j=1,2,...,s,\ s<\dim(V^{H}), \]
\ \ \ \ \ \ \ \ \ \ \ \ such that $\lambda_{i}^{\widetilde{H}}<\lambda_{i}^{H}$. Set $u_{i}^{0}=u_{i}^{\widetilde{H}}$, $\lambda_{i}^{0}=Rq(u^{0}_{i})$, $W^{0}=U^{0}=\text{span}\{u_{i}^{0}\}_{i=1}^{s},\ \Lambda^{0}=\{\lambda_{i}^{0}\}_{i=1}^{s}$.\\
$\bm{{\rm Step\ 2}}$ For $k=0, 1, 2, ...,$ solve \eqref{Equation_JacobiDavidsonCorrtion} inexactly through solving some parallel preconditioned\\ \ \ \ \ \ \ \ \ \ \ \ systems:
\begin{equation}\label{Preconditioneri}
t_{i}^{k+1}=(I-Q_{U^{k}})(B_{i}^{k})^{-1}r_{i}^{k}=(B_{i}^{k})^{-1}r_{i}^{k}-\sum_{i=1}^{s}b((B_{i}^{k})^{-1}r_{i}^{k},u_{i}^{k})u_{i}^{k},\ i=1,2,...,s,
\end{equation}
\ \ \ \ \ \ \ \ \ \ \ where $(B_{i}^{k})^{-1}$ is defined in \eqref{Preconditioner}.

$\bm{{\rm Step\ 3}}$ Solve the first $s$ eigenpairs in $W^{k+1}$:\\
\begin{equation}\label{UpdateRayleighRitz}
a(u_{i}^{k+1},v)=\lambda_{i}^{k+1}b(u_{i}^{k+1},v)\ \ \forall\ v\in W^{k+1},\ b(u_{i}^{k+1},u_{j}^{k+1})=\delta_{ij},
\end{equation}
\ \ \ \ \ \ \ \ \ \ \ where $\ i,j=1,2,...,s$, $W^{k+1}=W^{k}+\text{span}\{t_{i}^{k+1}\}_{i=1}^{s}$.\\
\ \ \ \ \ \ \ \ \ \ \ Set $U^{k+1}=\text{span}\{u_{i}^{k+1}\}_{i=1}^{s},\ \Lambda^{k+1}=\{\lambda_{i}^{k+1}\}_{i=1}^{s}.$\\
$\bm{{\rm Step\ 4}}$ If $\sum_{i=1}^{s}|\lambda_{i}^{k+1}-\lambda_{i}^{k}|<tol$, return $(\Lambda^{k+1},U^{k+1})$. Otherwise, goto $\bm{step\ 2}$.\\
\hline
\hline
\end{tabular}
\end{table}

\begin{remark}
Actually, the choice of $W^{k+1}$ may be different. We may choose
\begin{equation}\notag
W^{k+1}=V^{k}+{\rm span}\{t_{i}^{k+1}\}_{i=1}^{s},
\end{equation}
where $V^{k}$ is a smaller subspace satisfying $U^{k}\subset V^{k}$. For example, we may take $V^{k}:=U^{k}$ or $V^{k}:=U^{k-1}+U^{k}$. The advantage of these two is that $\dim{W^{k+1}}$ is independent of $k$, which may reduce the cost for solving the approximate eigenpairs in $W^{k+1}$.
\end{remark}

\begin{remark}
The meaning of the `block' in two-level BPJD method is understood as follows: Let $\sigma^{d}(v^{h})$ ($1\leq d\leq nd$) be the $d$th coordinate of $v^{h}\ (v^{h}\in V^{h})$ corresponding to the finite element basis. The operator $I-Q_{U^{k}}$ is represented through matrix form
$I_{nd\times nd}-X^{k}(X^{k})^{t}M$, where $M$ is the mass matrix corresponding to the finite element basis, the matrix $X^{k}$ is
\begin{gather*}
\begin{pmatrix}
\sigma^{1}(u_{1}^{k}) & \sigma^{1}(u_{2}^{k}) & ... & \sigma^{1}(u_{s}^{k})\\
\sigma^{2}(u_{1}^{k}) & \sigma^{2}(u_{2}^{k}) & ... & \sigma^{2}(u_{s}^{k})\\
...&...&...&...\\
\sigma^{nd}(u_{1}^{k}) & \sigma^{nd}(u_{2}^{k}) & ... & \sigma^{nd}(u_{s}^{k})\\
\end{pmatrix}_{nd\times s},
\end{gather*}
and $(X^{k})^{t}$ is the transpose of $X^{k}$. In particular, for $s=1$, the matrix version of the operator $I-Q_{u_{1}^{k}}$ is $I_{nd\times nd}-\sigma(u_{1}^{k})(\sigma(u_{1}^{k}))^{t}M$, where $\sigma(u_{1}^{k})$ is the transpose of $(\sigma^{1}(u_{1}^{k}),\sigma^{2}(u_{1}^{k}),...,\sigma^{nd}(u_{1}^{k}))$. So the block-version Jacobi-Davidson correction equations \eqref{Equation_JacobiDavidsonCorrtion} are solved at the same time by \eqref{Preconditioneri} for $i=1,2,...,s$.
\end{remark}

\begin{remark}
The purpose of Step 1 is to give an initial approximation for the proposed iteration algorithm. The condition $\lambda_{i}^{\widetilde{H}}<\lambda_{i}^{H}$, which can be achieved through $\widetilde{H}=\frac{H}{2^{\tau}}$ for any positive integer $\tau$, is to ensure that $(B_{0,i}^{0})^{-1}$ is well-defined in theoretical analysis. But we find that it is not necessary in practical computation.
\end{remark}
\begin{remark}\label{Remark_Functional}
For $s=1$, our algorithm may be regarded as a parallel preconditioned solver which solves
\begin{equation}\notag
u_{1}^{h}=\arg \min_{v\ne0, v\in V^{h}}\frac{a(v,v)}{b(v,v)}=\arg \min_{v\ne0, {\rm span}\{v\}\subset V^{h}} {\rm Tr} (Q_{v}A^{h}Q_{v}).
\end{equation}
For $s\geq 2$, our algorithm may be seen as a  parallel preconditioned solver which solves
\begin{equation}\notag
U_{s}^{h}=\arg \min_{U\subset V^{h}, \dim(U)=s}{\rm Tr} (Q_{U}A^{h}Q_{U}).
\end{equation}
\end{remark}
\par We may consider a functional $\mathcal{G}:\Sigma_{V_{h}}\to \mathcal{R}$ such that
$ \mathcal{G}(U)={\rm Tr}(Q_{U}A^{h}Q_{U})$ for all $U\in \Sigma_{V_{h}}$,
where $\Sigma_{V_{h}}$ includes all closed subspaces of $V^{h}$. Hence, we need to minimize the functional $\mathcal{G}(\cdot)$ in $\Xi_{V_{h}}:=\{U\in \Sigma_{V^{h}}\ |\ \dim{(U)}=s\}$ to obtain $(\{\lambda_{i}^{h}\}_{i=1}^{s},U_{s}^{h})$.

\section{Some properties of subspace method}
\par In this section, we present some useful properties in convergence analysis.
\par Since
$$U^{k}+{\rm span}\{t_{i}^{k+1}\}_{i=1}^{s}\subset W^{k}+\text{span}\{t_{i}^{k+1}\}_{i=1}^{s}=W^{k+1}\subset V^{h},$$ by the Courant-Fischer principle, we obtain
\begin{equation}\label{Lemma_EigenvalueDescent}
\lambda_{i}^{h}\leq \lambda_{i}^{k+1}\leq \lambda_{i}^{k},\ \ \ \ k=0,1,2,...,\ i=1,2,...,s.
\end{equation}
By Step 1 in Algorithm 3.1, we know that
$\lambda_{i}^{0}=\lambda_{i}^{\widetilde{H}}< \lambda_{i}^{H},\ i=1,2,...,s.$
Using \eqref{Lemma_EigenvalueDescent} and Assumption 1, we have
\begin{equation}\label{equation_EigenvalueRelation}
\lambda_{i}\leq \lambda_{i}^{h}\leq \lambda_{i}^{k}\leq \lambda_{i}^{0}=\lambda_{i}^{\widetilde{H}}<\lambda_{i}^{H}<\lambda_{s+1}^{h}\leq \lambda_{s+1}^{H},\ \ \ k=0,1,2,...,\ i=1,2,...,s.
\end{equation}
\par For our theoretical analysis, we define $(v,w)_{E_{i}^{k}}:=a(v,w)-\lambda^{k}_{i}b(v,w)$, $(v,w)_{E_{i}^{h}}:=a(v,w)-\lambda^{h}_{i}b(v,w)$ and $(v,w)_{E_{i}}:=a(v,w)-\lambda_{i}b(v,w)$ for all $v,w\in V^{h}$.
It is easy to check that the bilinear forms $(\cdot,\cdot)_{E_{i}^{k}}$, $(\cdot,\cdot)_{E_{i}^{h}}$ and $(\cdot,\cdot)_{E_{i}}$ construct inner products in $U_{s+1}^{h}$. The norms $||\cdot||_{E_{i}^{k}}$, $||\cdot||_{E_{i}^{h}}$ and $||\cdot||_{E_{i}}$ induced by $(\cdot,\cdot)_{E_{i}^{k}}$, $(\cdot,\cdot)_{E_{i}^{h}}$ and $(\cdot,\cdot)_{E_{i}}$, respectively, are equivalent to the norm $||\cdot||_{a}$ over $U_{s+1}^{h}$. In fact, on one hand, $(w,w)_{E_{i}^{k}}\leq (w,w)_{E_{i}^{h}}\leq (w,w)_{E_{i}} \leq a(w,w)$ for all $ w\in V^{h}.$ On the other hand, by \eqref{equation_EigenvalueRelation}, it is easy to get that $(v,v)_{E_{i}^{k}}\geq (\lambda_{s+1}^{h}-\lambda_{i}^{k})b(v,v)$ for all $v\in U_{s+1}^{h}$.
Moreover,
\begin{align*}
a(v,v)
=(v,v)_{E_{i}^{k}}+\lambda_{i}^{k}b(v,v)
\leq \beta(\lambda_{s+1}^{h})(v,v)_{E_{i}^{k}}\ \ \ \ \forall\ v\in U_{s+1}^{h},
\end{align*}
where the real valued function $\beta(\lambda):=1+\frac{\lambda_{i}^{k}}{\lambda-\lambda_{i}^{k}}$. Throughout the paper, we denote by $\mu_{i}^{k}:=(\lambda_{i}^{k})^{-1}$, $g_{i}^{k}:=\mu_{i}^{k}u_{i}^{k}-T^{h}u_{i}^{k}$, $Q_{\perp}^{k}:=I-Q_{U^{k}}$,  $P_{\perp}^{k}:=I-P_{U^k}$ and $e_{i,s+1}^{k}:=-Q_{s+1}^{h}u_{i}^{k}$, where $Q_{U^{k}}: V^{h}\to U^{k}$ denotes a $b(\cdot,\cdot)$-orthogonal projector and $P_{U^{k}}:V^{h}\to U^{k}$ denotes an $a(\cdot,\cdot)$-orthogonal projector.
\begin{lemma}\label{Lemma_NormQs+1h}
It holds that
$$||Q_{U^{k}}v||^{2}_{a}\leq \frac{1}{\lambda_{1}}\sum_{i=1}^{s}\lambda_{i}^{k}||v||_{a}^{2},\ \ \ ||Q_{\perp}^{k}v||^{2}_{a}\leq (2+\frac{2}{\lambda_{1}}\sum_{i=1}^{s}\lambda_{i}^{k})||v||_{a}^{2}\ \ \ \ \  \forall\ v\in V^{h}.$$
\end{lemma}
\begin{proof}
By the fact that $a(u_{i}^{k},u_{j}^{k})=\lambda_{i}^{k}b(u_{i}^{k},u_{j}^{k})=\lambda_{i}^{k}\delta_{ij}$, we know that $\{\sqrt{\mu_{i}^{k}}u_{i}^{k}\}_{i=1}^{s}$ forms a group of normal and orthogonal basis with respect to $a(\cdot,\cdot)$. For any $v\in V^{h}$, by the Cauchy-Schwarz inequality, we obtain
\begin{align*}
||Q_{U^{k}}v||_{a}^{2}
&=||\sum_{i=1}^{s}b(v,u_{i}^{k})u_{i}^{k}||_{a}^{2}
=||\sum_{i=1}^{s}b(v,\sqrt{\lambda_{i}^{k}}u_{i}^{k})\sqrt{\mu_{i}^{k}}u_{i}^{k}||_{a}^{2}
=\sum_{i=1}^{s}(b(v,\sqrt{\lambda_{i}^{k}}u_{i}^{k}))^{2}\\
&\leq ||v||_{b}^{2}\sum_{i=1}^{s}||\sqrt{\lambda_{i}^{k}}u_{i}^{k}||^{2}_{b}
\leq \frac{1}{\lambda_{1}}\sum_{i=1}^{s}\lambda_{i}^{k}||v||_{a}^{2}.
\end{align*}
Moreover,
$$
||Q_{\perp}^{k}v||_{a}^{2}
=||(I-Q_{U^{k}})v||_{a}^{2}
\leq 2\{||v||_{a}^{2}+||Q_{U^{k}}v||_{a}^{2}\}
\leq 2(1+\frac{1}{\lambda_{1}}\sum_{i=1}^{s}\lambda_{i}^{k})||v||_{a}^{2},
$$
which completes the proof of this lemma. \qed
\end{proof}

\begin{lemma}\label{Lemma_QUsQs1}
If Assumption \ref{Assumption1} holds, then for any $v^{H}\in V^{H}$, it holds that
\begin{equation}\label{Equation_QushQs1H}
||Q_{s}^{h}Q_{s+1}^{H}v^{H}||_{b}\leq CH^{2}||Q_{s+1}^{H}v^{H}||_{b},\ \ \ ||Q_{s}^{h}Q_{s+1}^{H}v^{H}||_{a}\leq CH||Q_{s+1}^{H}v^{H}||_{a}.
\end{equation}
\end{lemma}
\begin{proof}
By Remark \ref{Remark_Hausdorff}, Theorem \ref{TheoremPriorLaplacian} and the fact $\dim(U_{s}^{h})=\dim(U_{s}^{H})=\dim(U_{s})$, we have
\begin{equation}\label{Equation_sinb}
\sin_{b}\{U_{s}^{h},U_{s}^{H}\}\leq \sin_{b}\{U_{s}^{h},U_{s}\}+\sin_{b}\{U_{s},U_{s}^{H}\}\leq CH^{2}.
\end{equation}
If $Q_{s}^{h}Q_{s+1}^{H}v^{H}=0$, \eqref{Equation_QushQs1H} holds.
For $Q_{s}^{h}Q_{s+1}^{H}v^{H}(\ne 0)\in U_{s}^{h}$, take $\widetilde{w}_{s}^{h}=\frac{Q_{s}^{h}Q_{s+1}^{H}v^{H}}{||Q_{s}^{h}Q_{s+1}^{H}v^{H}||_{b}}$. By \eqref{Equation_sinb}, there exists a $v_{s}^{H}(\ne0)\in U_{s}^{H}$ such that
\begin{equation}\notag
||\widetilde{w}_{s}^{h}-v_{s}^{H}||_{b}\leq CH^{2}.
\end{equation}
Then we have,
\begin{equation}\notag
||\widetilde{w}_{s}^{h}-\frac{v_{s}^{H}}{||v_{s}^{H}||_{b}}||_{b}=
||\frac{\widetilde{w}_{s}^{h}}{||\widetilde{w}_{s}^{h}||_{b}}-\frac{v_{s}^{H}}{||v_{s}^{H}||_{b}}||_{b}\leq \frac{2}{||\widetilde{w}_{s}^{h}||_{b}}||\widetilde{w}_{s}^{h}-v_{s}^{H}||_{b}\leq CH^{2}.
\end{equation}
Taking $w_{s}^{H}=\frac{||Q_{s}^{h}Q_{s+1}^{H}v^{H}||_{b}}{||v_{s}^{H}||_{b}}v_{s}^{H}\in U_{s}^{H}$, we get
$$||Q_{s}^{h}Q_{s+1}^{H}v^{H}-w_{s}^{H}||_{b}\leq CH^{2}||Q_{s}^{h}Q_{s+1}^{H}v^{H}||_{b}.$$
Moreover,
\begin{align*}
||Q_{s}^{h}Q_{s+1}^{H}v^{H}||^{2}_{b}
=b(Q_{s}^{h}Q_{s+1}^{H}v^{H}-w_{s}^{H},Q_{s+1}^{H}v^{H})
\leq CH^{2}||Q_{s}^{h}Q_{s+1}^{H}v^{H}||_{b}||Q_{s+1}^{H}v^{H}||_{b},
\end{align*}
which yields the first inequality in \eqref{Equation_QushQs1H}.
Similarly, we may prove that $||Q_{s}^{h}Q_{s+1}^{H}v^{H}||_{a}\leq CH||Q_{s+1}^{H}v^{H}||_{a}$, and then obtain the proof of this lemma. \qed
\end{proof}

\par By the analysis above, for any $v^{H}\in V^{H}$, we could easily estimate $||Q_{s+1}^{h}Q_{s}^{H}v^{H}||_{b}$ and $||Q_{s+1}^{h}Q_{s}^{H}v^{H}||_{a}$, similarly. For any $v^{h}\in V^{h}$, we have estimate results as follows
\begin{equation}\label{Equation_QsHQHQs1hvh}
||Q_{s}^{H}Q^{H}Q_{s+1}^{h}v^{h}||_{b}\leq CH^{2}||Q_{s+1}^{h}v^{h}||_{b},
\end{equation}
and
\begin{equation}\label{Equation_Qs1HQHQshvh}
||Q_{s+1}^{H}Q^{H}Q_{s}^{h}v^{h}||_{b}\leq CH^{2}||Q_{s}^{h}v^{h}||_{b}.
\end{equation}

\begin{lemma}\label{Lemma_gik}
Let $a(u_{i}^{k},u_{j}^{k})=\lambda_{i}^{k}b(u_{i}^{k},u_{j}^{k})=\lambda_{i}^{k}\delta_{ij}$ and Assumption \ref{Assumption1} hold. It holds that
\begin{equation}\label{gik}
\sum_{i=1}^{s}a(g_{i}^{k},g_{i}^{k}) \leq \sum_{i=1}^{s}(\mu_{i}^{h}-\mu_{i}^{k}).
\end{equation}
In particular,
$\sum_{i=1}^{s}||g_{i}^{k}||_{a}^{2}\leq CH^{2}.$
\end{lemma}
\begin{proof}
As $a(u_{i}^{k},u_{i}^{k})=\lambda_{i}^{k}b(u_{i}^{k},u_{i}^{k})=\lambda_{i}^{k}a(T^{h}u_{i}^{k},u_{i}^{k})=\lambda_{i}^{k}$, we have
\begin{align}
a(g_{i}^{k},g_{i}^{k})
=a(T^{h}u_{i}^{k},T^{h}u_{i}^{k})-\mu_{i}^{k}
=a(T^{h}(T^{h})^{\frac{1}{2}}u_{i}^{k},(T^{h})^{\frac{1}{2}}u_{i}^{k})-\mu_{i}^{k}.\label{align_ealpha1}
\end{align}
Since $a((T^{h})^{\frac{1}{2}}u_{i}^{k},(T^{h})^{\frac{1}{2}}u_{j}^{k})=\delta_{ij}$, we may consider the eigenvalue problem $a(T^{h}w_{i},v)=\nu_{i}^{k}a(w_{i},v)$  for all $v\in$ span$\{(T^{h})^{\frac{1}{2}}u_{i}^{k}\}_{i=1}^{s}$.
Moreover,
\begin{equation}\label{equation_traceOk}
\sum_{i=1}^{s}\nu_{i}^{k}
=\sum_{i=1}^{s}\frac{a(T^{h}(T^{h})^{\frac{1}{2}}u_{i}^{k},(T^{h})^{\frac{1}{2}}u_{i}^{k})}
{a((T^{h})^{\frac{1}{2}}u_{i}^{k},(T^{h})^{\frac{1}{2}}u_{i}^{k})}
=\sum_{i=1}^{s}Rt((T^{h})^{\frac{1}{2}}u_{i}^{k}).
\end{equation}
As span$\{(T^{h})^{\frac{1}{2}}u_{i}^{k}\}_{i=1}^{s}\subset V^{h}$, we know $\nu_{i}^{k}\leq \mu_{i}^{h}$, which, together with \eqref{align_ealpha1}, \eqref{equation_traceOk}, yields
\begin{align*}
\sum_{i=1}^{s}a(g_{i}^{k},g_{i}^{k})
&=\sum_{i=1}^{s}\{Rt((T^{h})^{\frac{1}{2}}u_{i}^{k})a((T^{h})^{\frac{1}{2}}u_{i}^{k},
(T^{h})^{\frac{1}{2}}u_{i}^{k})-\mu_{i}^{k}\}\notag\\
&=\sum_{i=1}^{s}\{Rt((T^{h})^{\frac{1}{2}}u_{i}^{k})-\mu_{i}^{k}\}
\leq \sum_{i=1}^{s}(\mu_{i}^{h}-\mu_{i}^{k}).
\end{align*}
By \eqref{equation_EigenvalueRelation} and Theorem \ref{TheoremPriorLaplacian}, we get
\begin{align*}
&\ \ \ \ \sum_{i=1}^{s}a(g_{i}^{k},g_{i}^{k})\leq  \sum_{i=1}^{s}(\mu_{i}^{h}-\mu_{i}^{k})
\leq C\sum_{i=1}^{s}(\lambda_{i}^{k}-\lambda_{i}^{h})
\leq C\sum_{i=1}^{s}(\lambda_{i}^{H}-\lambda_{i})\leq CH^{2},
\end{align*}
which completes the proof of this lemma.  \qed
\end{proof}

\par The following lemma illustrates that the gap between $U_{s}^{h}$ and $U^{k}$ with respect to $||\cdot||_{b}$ is bounded by
the total error of eigenvalues. In particular, it is bounded by $CH^{2}$.
\begin{lemma}\label{Lemma_theta_bk}
Let $a(u_{i}^{k},u_{j}^{k})=\lambda_{i}^{k}b(u_{i}^{k},u_{j}^{k})=\lambda_{i}^{k}\delta_{ij}$ and Assumption \ref{Assumption1} hold, then
\begin{equation}\label{EstimateThetbb}
(\theta_{b}^{k})^{2}\leq \frac{1}{\lambda_{s+1}^{h}-\lambda_{s}^{h}}\sum_{i=1}^{s}(\lambda_{i}^{k}-\lambda_{i}^{h}),
\end{equation}
where $\theta_{b}^{k}$ is the gap between $U_{s}^{h}$ and $U^{k}$ with respect to $||\cdot||_{b}$.
In particular, $(\theta_{b}^{k})^{2}\leq CH^{2}.$

\end{lemma}
\begin{proof}
Since $a(u_{i}^{k},u_{j}^{k})=\lambda_{i}^{k}b(u_{i}^{k},u_{j}^{k})=\lambda_{i}^{k}\delta_{ij}$, it is easy to know that $\{u_{i}^{k}\}_{i=1}^{s}$ forms a group of normal and orthogonal basis for $U^{k}$ with respect to $b(\cdot,\cdot)$. By \eqref{VhDecomposition}, we get
\begin{align}
\lambda_{s+1}^{h}-\lambda_{i}^{k}
&=b((\lambda_{s+1}^{h}-A^{h})u_{i}^{k},u_{i}^{k})
=b((\lambda_{s+1}^{h}-A^{h})Q_{s}^{h}u_{i}^{k},Q_{s}^{h}u_{i}^{k})\notag\\
&\ \ \ \ +b((\lambda_{s+1}^{h}-A^{h})Q_{s+1}^{h}u_{i}^{k},Q_{s+1}^{h}u_{i}^{k})
\leq b((\lambda_{s+1}^{h}-A^{h})Q_{s}^{h}u_{i}^{k},Q_{s}^{h}u_{i}^{k}).\label{ThetaEstimate1}
\end{align}
Combining $\eqref{ThetaEstimate1}$ and the fact that $Q_{s}^{h}v=\sum_{j=1}^{s}Q_{j,s}^{h}v$ for all $v\in V^{h}$, we have
\begin{align*}
 \sum_{i=1}^{s}(\lambda_{s+1}^{h}-\lambda_{i}^{k})
&\leq \sum_{i=1}^{s}\sum_{j=1}^{s}(\lambda_{s+1}^{h}-\lambda_{j}^{h})b(Q_{j,s}^{h}u_{i}^{k},Q_{j,s}^{h}u_{i}^{k})
=\sum_{i=1}^{s}\sum_{j=1}^{s}(\lambda_{s+1}^{h}-\lambda_{j}^{h})(b(u_{i}^{k},u_{j}^{h}))^{2}\\
&=\sum_{j=1}^{s}(\lambda_{s+1}^{h}-\lambda_{j}^{h})\sum_{i=1}^{s}(b(u_{i}^{k},u_{j}^{h}))^{2}
=\sum_{j=1}^{s}(\lambda_{s+1}^{h}-\lambda_{j}^{h})||Q_{U^{k}}u_{j}^{h}||^{2}_{b}\\
&=\sum_{j=1}^{s}(\lambda_{s+1}^{h}-\lambda_{j}^{h})(1-||(I-Q_{U^{k}})u_{j}^{h}||^{2}_{b})
=\sum_{j=1}^{s}(\lambda_{s+1}^{h}-\lambda_{j}^{h})(1-\sin^{2}_{b}\{u_{j}^{h};U^{k}\}).
\end{align*}
Hence, $\sum_{i=1}^{s}(\lambda_{s+1}^{h}-\lambda_{i}^{h})\sin_{b}^{2}\{u_{i}^{h};U^{k}\}\leq \sum_{i=1}^{s}(\lambda_{i}^{k}-\lambda_{i}^{h}).$
If  Assumption 1 holds (also $\lambda_{s}^{h}<\lambda_{s+1}^{h}$ as $h\to 0$), we obtain
\begin{equation}\notag
\sum_{i=1}^{s}\sin_{b}^{2}\{u_{i}^{h};U^{k}\}\leq \frac{1}{\lambda_{s+1}^{h}-\lambda_{s}^{h}}\sum_{i=1}^{s}(\lambda_{i}^{k}-\lambda_{i}^{h}),
\end{equation}
 which, together with  the fact that $(\theta^{k}_{b})^{2}=\sin_{b}^{2}\{U_{s}^{h};U^{k}\}\leq\sum_{i=1}^{s}\sin_{b}^{2}\{u_{i}^{h};U^{k}\}$ (see Lemma 3.4 in \cite{Bramble2} or Corollary 2.2 in \cite{knyazev2006new}), yields \eqref{EstimateThetbb}.
By \eqref{equation_EigenvalueRelation} and Theorem \ref{TheoremPriorLaplacian}, we get
\begin{align*}
(\theta_{b}^{k})^{2}
\leq \frac{1}{\lambda_{s+1}^{h}-\lambda_{s}^{h}}\sum_{i=1}^{s}(\lambda_{i}^{k}-\lambda_{i}^{h})
\leq C\sum_{i=1}^{s}(\lambda_{i}^{H}-\lambda_{i})
\leq CH^{2},
\end{align*}
which completes the proof of this lemma.  \qed
\end{proof}
\par We may use a similar argument as in the proof of Lemma \ref{Lemma_theta_bk} to obtain the following result.
\begin{corollary}\label{Lemma_theta_ak}
Let $a(u_{i}^{k},u_{j}^{k})=\lambda_{i}^{k}b(u_{i}^{k},u_{j}^{k})=\lambda_{i}^{k}\delta_{ij}$ and Assumption \ref{Assumption1} hold, then
\begin{equation}\label{EstimateThetaa}
(\theta_{a}^{k})^{2}\leq \frac{1}{\mu_{s}^{h}-\mu_{s+1}^{h}}\sum_{i=1}^{s}(\mu_{i}^{h}-\mu_{i}^{k}),
\end{equation}
where $\theta_{a}^{k}$ is the gap between $U_{s}^{h}$ and $U^{k}$ with respect to $||\cdot||_{a}$.
In particular, $(\theta_{a}^{k})^{2}\leq CH^{2}.$
\end{corollary}

\par The gap $\theta_{a}^{k}$ and $\theta_{b}^{k}$ can be characterized by the $a$-norm and the $b$-norm of the operators, respectively.
\begin{lemma}\label{Lemma_Theta_EnergyL2}
It holds that
\begin{equation}\notag
\theta^{k}_{a}=\interleave P_{\perp}^{k}Q_{s}^{h}\interleave_{a}
=\interleave Q_{s}^{h}P_{\perp}^{k}\interleave_{a}
=\interleave Q_{s+1}^{h}P_{U^{k}}\interleave_{a}
=\interleave P_{U^{k}}Q_{s+1}^{h}\interleave_{a},
\end{equation}
and
\begin{equation}\notag
\theta^{k}_{b}=\interleave Q_{\perp}^{k}Q_{s}^{h}\interleave_{b}
=\interleave Q_{s}^{h}Q_{\perp}^{k}\interleave_{b}
=\interleave Q_{s+1}^{h}Q_{U^{k}}\interleave_{b}
=\interleave Q_{U^{k}}Q_{s+1}^{h}\interleave_{b}.
\end{equation}
\end{lemma}

\begin{proof}
Combining Definition \ref{Definition_gap} and Remark \ref{Remark_Hausdorff}, it is easy to prove this lemma. In order to focus on our main theoretical analysis, we ignore this proof here. \qed
\end{proof}

\begin{lemma}\label{Lemma_Relationu1e2}
If Assumption \ref{Assumption1} holds, then
\begin{align}\label{Relationu1e2_equlity}
||e_{i,s+1}^{k}||_{E_{i}^{k}}^{2}=(\lambda_{i}^{k}-Rq(Q_{s}^{h}u_{i}^{k}))||Q_{s}^{h}u_{i}^{k}||_{b}^{2},
\end{align}
\begin{equation}\label{Relationu1e2_inequlity}
0 \leq \lambda_{i}^{k}-Rq(Q_{s}^{h}u_{i}^{k})\leq CH^{2},
\end{equation}
and
\begin{equation}\label{Relationu1e2_inequlity3}
\lambda_{i}^{k}-Rq(Q_{s}^{h}u_{i}^{k})\leq CH||g_{i}^{k}||_{a}.
\end{equation}
Moreover,
\begin{equation}\label{Relationu1e2_inequlity2}
||(\lambda_{i}^{k}-A^{h})Q_{s}^{h}u_{i}^{k}||_{b}\leq C||g_{i}^{k}||_{a}.
\end{equation}
\end{lemma}
\begin{proof}
We first prove \eqref{Relationu1e2_equlity} and \eqref{Relationu1e2_inequlity}. Using \eqref{UpdateRayleighRitz}, we have $a(u_{i}^{k},u_{i}^{k})=\lambda_{i}^{k}b(u_{i}^{k},u_{i}^{k})$. Therefore,
\begin{align}
&\ \ \ \ ||e_{i,s+1}^{k}||_{E_{i}^{k}}^{2}
=b((\lambda_{i}^{k}-A^{h})Q_{s}^{h}u_{i}^{k},Q_{s}^{h}u_{i}^{k})
=(\lambda_{i}^{k}-Rq(Q_{s}^{h}u_{i}^{k}))||Q_{s}^{h}u_{i}^{k}||_{b}^{2},  \label{RelationErrorEigen}
\end{align}
which means that \eqref{Relationu1e2_equlity} holds and $\lambda_{i}^{k}-Rq(Q_{s}^{h}u_{i}^{k})\geq 0$. By Lemma \ref{Lemma_theta_bk} and Lemma \ref{Lemma_Theta_EnergyL2}, we have
\begin{align}
&\ \ \ \ (\lambda_{i}^{k}-Rq(Q_{s}^{h}u_{i}^{k}))||Q_{s}^{h}u_{i}^{k}||_{b}^{2}=
(\lambda_{i}^{k}-Rq(Q_{s}^{h}u_{i}^{k}))(1-||Q_{s+1}^{h}Q_{U^{k}}u_{i}^{k}||_{b}^{2})\notag\\
&\geq (\lambda_{i}^{k}-Rq(Q_{s}^{h}u_{i}^{k}))(1-(\theta^{k}_{b})^{2}||u_{i}^{k}||_{b}^{2})
\geq (\lambda_{i}^{k}-Rq(Q_{s}^{h}u_{i}^{k}))(1-CH^{2}).\label{equation_lowerestimate}
\end{align}
By \eqref{RelationErrorEigen}, Corollary \ref{Lemma_theta_ak} and Lemma \ref{Lemma_Theta_EnergyL2}, we obtain
\begin{align}
&\ \ \ \ (\lambda_{i}^{k}-Rq(Q_{s}^{h}u_{i}^{k}))||Q_{s}^{h}u_{i}^{k}||_{b}^{2}
=||e_{i,s+1}^{k}||_{E_{i}^{k}}^{2}
\leq ||Q_{s+1}^{h}u_{i}^{k}||_{a}^{2}
= ||Q_{s+1}^{h}P_{U^{k}}u_{i}^{k}||_{a}^{2}\leq CH^{2}.\label{align_Lemma_upperbound}
\end{align}
Using \eqref{equation_lowerestimate} and \eqref{align_Lemma_upperbound}, we get \eqref{Relationu1e2_inequlity}.
\par By Corollary \ref{Lemma_theta_ak}, Lemma \ref{Lemma_Theta_EnergyL2}, and the facts $P_{\perp}^{k}g_{i}^{k}=g_{i}^{k}$ and $r_{i}^{k}=-\lambda_{i}^{k}A^{h}g_{i}^{k}$, we deduce
\begin{align*}
(\lambda_{i}^{k}-Rq(Q_{s}^{h}u_{i}^{k}))||Q_{s}^{h}u_{i}^{k}||_{b}^{2}
&=b((\lambda_{i}^{k}-A^{h})Q_{s}^{h}u_{i}^{k},Q_{s}^{h}u_{i}^{k})
=-\lambda_{i}^{k}a(g_{i}^{k},Q_{s}^{h}u_{i}^{k})\notag\\
&=-\lambda_{i}^{k}a(Q_{s}^{h}P_{\perp}^{k}g_{i}^{k},Q_{s}^{h}u_{i}^{k})
\leq C\theta_{a}^{k}||g_{i}^{k}||_{a}||Q_{s}^{h}u_{i}^{k}||_{a}\leq CH||g_{i}^{k}||_{a},
\end{align*}
which, together with \eqref{equation_lowerestimate}, yields \eqref{Relationu1e2_inequlity3}.
\par Since $A^{h}|_{U_{s}^{h}}:U_{s}^{h}\to U_{s}^{h}$ is a linear isomorphism, we know that
there exists an unique $v_{i,s}^{k}\in U_{s}^{h}$ such that $A^{h}v_{i,s}^{k}=(\lambda_{i}^{k}-A^{h})Q_{s}^{h}u_{i}^{k}.$
Accordingly,
\begin{align*}
&\ \ \ \ ||(\lambda_{i}^{k}-A^{h})Q_{s}^{h}u_{i}^{k}||_{b}^{2}
=b(A^{h}(A^{h})^{\frac{1}{2}}v_{i,s}^{k},(A^{h})^{\frac{1}{2}}v_{i,s}^{k})
\leq \lambda_{s}^{h}a(v_{i,s}^{k},v_{i,s}^{k})\\
&=\lambda_{s}^{h}(\lambda_{i}^{k})^{2}||(\mu_{i}^{k}-T^{h})Q_{s}^{h}u_{i}^{k}||_{a}^{2}
=\lambda_{s}^{h}(\lambda_{i}^{k})^{2}||Q_{s}^{h}g_{i}^{k}||^{2}_{a}
\leq \lambda_{s}^{h}(\lambda_{i}^{k})^{2}||g_{i}^{k}||_{a}^{2}.
\end{align*}
This leads to \eqref{Relationu1e2_inequlity2}. \qed
\end{proof}
\section{Convergence analysis}
\par In this section, we focus on giving a rigorous convergence analysis for the two-level BPJD method. We first present the main theoretical result in this paper. The rest of this section is organized as follows: In subsection 5.1, by choosing a suitable coarse component and using some overlapping DD techniques, we deduce the error reduction from $e_{i,s+1}^{k}$ to $\widetilde{e}_{i,s+1}^{k+1}$ ($\widetilde{e}_{i,s+1}^{k+1}:=-Q_{s+1}^{h}\widetilde{u}_{i}^{k+1}$, where $\widetilde{u}_{i}^{k+1}$ shall be defined in \eqref{uik1}, $i=1,2,...,s$). In subsection 5.2, by constructing an auxiliary eigenvalue problem in span$\{\widetilde{u}_{i}^{k+1}\}_{i=1}^{s}$ which shall be presented in \eqref{AuxiliaryEigenvalueProblem}, we may establish the total error reduction of the first $s$ eigenvalues.

\begin{theorem}\label{MainResult}
Assume that Assumption \ref{Assumption1} and Assumption \ref{Assumption2} hold, then
\begin{equation}\label{Total_Theorem}
\sum_{i=1}^{s}(\lambda_{i}^{k+1}-\lambda_{i}^{h})\leq \gamma \sum_{i=1}^{s}(\lambda_{i}^{k}-\lambda_{i}^{h}),
\end{equation}
and
\begin{equation}\label{Hausdorffa_Theorem}
(\theta^{k}_{a})^{2}\leq C\gamma^{k},
\end{equation}
\begin{equation}\label{Hausdorffb_Theorem}
(\theta^{k}_{b})^{2}\leq C\gamma^{k}.
\end{equation}
Here $\gamma=c(H)(1-C\frac{\delta^{2m-1}}{H^{2m-1}})^{2},\ m=1,2$. The constant $C$ is independent of $h,\ H,\ \delta$ and the left gaps of eigenvalues $\{\lambda_{i}\}_{i=2}^{s}$, and the $H$-dependent constant $c(H)\ (=1+\frac{CH}{(1-C\frac{\delta^{2m-1}}{H^{2m-1}})^{2}})$ decreases monotonically to $1$, as $H\to 0$.
\end{theorem}
\begin{remark}
In order to make the main idea of the proof of Theorem \ref{MainResult} clear, we only consider the model problem \eqref{Laplacian}. The proof for problem \eqref{Biharmonic} is similar.
\end{remark}
\par  For the convenience of the following convergence analysis, we first choose some special functions defined as
\begin{equation}\label{uik1}
\widetilde{u}_{i}^{k+1}:=u_{i}^{k}+\alpha_{i}^{k}t_{i}^{k+1}\in U^{k}+\text{span}\{t_{i}^{k+1}\}_{i=1}^{s}\subset W^{k+1},\ \ i=1,2,...,s,
\end{equation}
to analyze the error reduction, where $\alpha_{i}^{k}\ (i=1,2,...,s)$ are some undetermined parameters dependent on $N_{0}$. From \eqref{Preconditioneri} and \eqref{uik1}, we know
\begin{equation}\label{Equation_construction_utilde}
\widetilde{u}_{i}^{k+1}=u_{i}^{k}+\alpha_{i}^{k}Q_{\perp}^{k}(B_{i}^{k})^{-1}r_{i}^{k},\ \ i=1,2,...,s,
\end{equation}
which are linearly independent. So we may construct an auxiliary eigenvalue problem in span$\{\widetilde{u}_{i}^{k+1}\}_{i=1}^{s}$:
\begin{equation}\label{AuxiliaryEigenvalueProblem}
a(\hat{u}_{i}^{k+1},v)=\hat{\lambda}_{i}^{k+1}b(\hat{u}_{i}^{k+1},v)\ \ \ \ \forall\ v \in \widetilde{U}^{k+1}:=\text{span}\{\widetilde{u}_{i}^{k+1}\}_{i=1}^{s}.
\end{equation}
Since $\widetilde{U}^{k+1}\subset W^{k+1}$, it is easy to see that $\lambda_{i}^{k+1}\leq \hat{\lambda}_{i}^{k+1}$.
\par The idea of the proof of Theorem $\ref{MainResult}$ is to design an auxiliary eigenvalue problem \eqref{AuxiliaryEigenvalueProblem}. By the `bridge' term $\mathcal{G}(\widetilde{U}^{k+1})-\mathcal{G}(U_{s}^{h})$, we may obtain the error reduction from  $\mathcal{G}(U^{k})-\mathcal{G}(U_{s}^{h})$ to $\mathcal{G}(U^{k+1})-\mathcal{G}(U_{s}^{h})$, i.e., by the `bridge' term $\sum_{i=1}^{s}(\hat{\lambda}_{i}^{k+1}-\lambda_{i}^{h})$, we may obtain the total error reduction from  $\sum_{i=1}^{s}(\lambda_{i}^{k}-\lambda_{i}^{h})$ to $\sum_{i=1}^{s}(\lambda_{i}^{k+1}-\lambda_{i}^{h})$.

\subsection{The error from the new DD preconditioner}
\par The block-version Jacobi-Davidson correction equations \eqref{Equation_JacobiDavidsonCorrtion} are solved inexactly, i.e., $$t_{i}^{k+1}=Q_{\perp}^{k}(B_{i}^{k})^{-1}r_{i}^{k}.$$
From \eqref{Equation_construction_utilde}, we first analyze the error reduction from $e_{i,s+1}^{k}(:=-Q_{s+1}^{h}u_{i}^{k})$ to $\widetilde{e}_{i,s+1}^{k+1}(:=-Q_{s+1}^{h}\widetilde{u}_{i}^{k+1})$. The orthogonal projection $-Q_{s+1}^{h}$ is applied to both sides of \eqref{Equation_construction_utilde}, we obtain
\begin{equation}\label{DefineEis1k}
\widetilde{e}_{i,s+1}^{k+1}=e_{i,s+1}^{k}-\alpha_{i}^{k}Q_{s+1}^{h}Q_{\perp}^{k}(B_{i}^{k})^{-1}r_{i}^{k}.
\end{equation}
Moreover, by the splitting of the identity operator on $V^{h}$ corresponding to \eqref{VhDecomposition}, we deduce
\begin{align}
\widetilde{e}_{i,s+1}^{k+1}
&=e_{i,s+1}^{k}-\alpha_{i}^{k}Q_{s+1}^{h}(B_{i}^{k})^{-1}r_{i}^{k}+\alpha_{i}^{k}Q_{s+1}^{h}Q_{U^{k}}(B_{i}^{k})^{-1}r_{i}^{k}\notag\\
&=e_{i,s+1}^{k}-\alpha_{i}^{k}Q_{s+1}^{h}(B_{i}^{k})^{-1}(\lambda_{i}^{k}-A^{h})(Q_{s}^{h}u_{i}^{k}-
e_{i,s+1}^{k})+\alpha_{i}^{k}Q_{s+1}^{h}Q_{U^{k}}(B_{i}^{k})^{-1}r_{i}^{k}\notag\\
&=\{e_{i,s+1}^{k}+\alpha_{i}^{k}Q_{s+1}^{h}(B_{i}^{k})^{-1}(\lambda_{i}^{k}-A^{h})e_{i,s+1}^{k}\}
+\alpha_{i}^{k}\{Q_{s+1}^{h}(B_{i}^{k})^{-1}(A^{h}-\lambda_{i}^{k})Q_{s}^{h}u_{i}^{k}\notag\\
&\ \ \ \ +Q_{s+1}^{h}Q_{U^{k}}(B_{i}^{k})^{-1}r_{i}^{k}\}=:I^{k}_{1,i}+I^{k}_{2,i}.\label{SumI1ikI2ik}
\end{align}
\par For simplicity, we define $G_{i}^{k}:=I+\alpha_{i}^{k}Q_{s+1}^{h}(B_{i}^{k})^{-1}(\lambda_{i}^{k}-A^{h})$. It is easy to see that $I^{k}_{1,i}=G_{i}^{k}e_{i,s+1}^{k}$. In this paper, we call $I^{k}_{1,i}$ the principal error term and $G_{i}^{k}:U_{s+1}^{h}\to U_{s+1}^{h}$ the principal error operator. Meanwhile, we call $I^{k}_{2,i}$ the additional error term.

\subsubsection{Estimate of the principal error term $I^{k}_{1,i}$}
\par In this subsection, we shall use the theory of the two-level domain decomposition method to analyze the principal error term $I^{k}_{1,i}$. Actually, we only need to estimate the spectral radius of the principal error operator $G_{i}^{k}$.
\begin{theorem}\label{Gkv2theorem}
Let Assumption \ref{Assumption1} and Assumption \ref{Assumption2} hold. Then for sufficiently small $\alpha_{i}^{k}$,
\begin{equation}\label{Gkv2hestimate}
||G_{i}^{k}v||_{E_{i}^{k}}\leq (1-C\frac{\delta}{H})||v||_{E_{i}^{k}}\ \ \ \ \ \forall \ v\in U_{s+1}^{h},\ \ i=1,2,...,s.
\end{equation}
\end{theorem}
\par First of all, we give two useful lemmas. The first lemma (Lemma \ref{Lemma_SymmetricPositive}) illustrates that the principal error operator $G_{i}^{k}:U_{s+1}^{h}\to U_{s+1}^{h}$ is symmetric and positive definite with respect to $(\cdot,\cdot)_{E_{i}^{k}}$. The second lemma (Lemma \ref{Lemma_StableDecomposition}) gives a stable spacial decomposition for the error subspace $U_{s+1}^{h}$ instead of the whole space $V^{h}$. Hence, the constructions of both coarse component and local fine components in this paper are different from those in \cite{ToselliM}.
\begin{lemma}\label{Lemma_SymmetricPositive}
Under the same assumptions as in Theorem \ref{Gkv2theorem}, for any $i\ (i=1,2,...,s)$, the operator $G_{i}^{k}:U_{s+1}^{h}\to U_{s+1}^{h}$ is symmetric with respect to $(\cdot,\cdot)_{E_{i}^{k}}$. Furthermore, if $\alpha_{i}^{k}$ is sufficiently small, the operator $G_{i}^{k}:U_{s+1}^{h}\to U_{s+1}^{h}$ is positive definite.
\end{lemma}
\begin{proof}
We first prove that the operator $G_{i}^{k}:U_{s+1}^{h}\to U_{s+1}^{h}$ is symmetric with respect to $(\cdot,\cdot)_{E_{i}^{k}}$. Since the operators $(B_{0,i}^{k})^{-1}$ and $(B_{l,i}^{k})^{-1} (l=1,2,...,N)$ are symmetric with respect to $b(\cdot,\cdot)$, we have
\begin{align}
&\ \ \ \ (Q_{s+1}^{h}(B_{i}^{k})^{-1}(\lambda_{i}^{k}-A^{h})v,w)_{E_{i}^{k}}
=((B_{i}^{k})^{-1}(\lambda_{i}^{k}-A^{h})v,w)_{E_{i}^{k}}\notag\\
&=b((A^{h}-\lambda_{i}^{k})v,(B_{i}^{k})^{-1}(\lambda_{i}^{k}-A^{h})w)
=(v,Q_{s+1}^{h}(B_{i}^{k})^{-1}(\lambda_{i}^{k}-A^{h})w)_{E_{i}^{k}},\ \ \ \ \forall\ v,w\in U_{s+1}^{h},
\end{align}
which means that the operator $G_{i}^{k}:U_{s+1}^{h}\to U_{s+1}^{h}$ is symmetric with respect to $(\cdot,\cdot)_{E_{i}^{k}}$.
\par  Next, for any $i\ (i=1,2,...,s)$, define an operator $T_{0,i}^{k}:U_{s+1}^{h}\to U_{s+1}^{H}$ such that for any $v\in U_{s+1}^{h}$,
\begin{equation}\label{Definition_T0ik}
(T_{0,i}^{k}v,w)_{E_{i}^{k}}=(v,w)_{E_{i}^{k}} \ \ \ \ \ \ \forall\ w\in U_{s+1}^{H}.
\end{equation}
By \eqref{equation_EigenvalueRelation} and the Lax-Milgram Theorem, we know that the operator $T_{0,i}^{k}$ is well-defined. Similarly, we may define some operators $T_{l,i}^{k}:U_{s+1}^{h}\to V^{(l)}\ (l=1,2,...,N)$ such that for any $v\in U_{s+1}^{h}$,
\begin{equation}\label{Definition_Tlik}
(T_{l,i}^{k}v,w)_{E_{i}^{k}}=(v,w)_{E_{i}^{k}} \ \ \ \ \ \  \forall\ w\in V^{(l)}.
\end{equation}
By \eqref{EigenvalueLocal} and the Lax-Milgram Theorem, the operators $T_{l,i}^{k}\ (l=1,2,...,N)$ are also well-defined.
It is easy to check that $T_{0,i}^{k}=(B_{0,i}^{k})^{-1}Q_{s+1}^{H}Q^{H}(A^{h}-\lambda_{i}^{k})$ and $T_{l,i}^{k}=(B_{l,i}^{k})^{-1}Q^{(l)}(A^{h}-\lambda_{i}^{k})$. Moreover,
$$G_{i}^{k}=I+\alpha_{i}^{k}Q_{s+1}^{h}(B_{i}^{k})^{-1}(\lambda_{i}^{k}-A^{h})=I-\alpha_{i}^{k} Q_{s+1}^{h}T_{0,i}^{k}-\alpha_{i}^{k} \sum_{l=1}^{N}Q_{s+1}^{h}T_{l,i}^{k}.$$
For any $v\in U_{s+1}^{h},$ by \eqref{Definition_T0ik} and \eqref{Definition_Tlik}, we have
\begin{align}
(G_{i}^{k}v,v)_{E_{i}^{k}}
&=||v||_{E_{i}^{k}}^{2}-\alpha_{i}^{k}||T_{0,i}^{k}v||_{E_{i}^{k}}^{2}
-\alpha_{i}^{k}\sum_{l=1}^{N}||T_{l,i}^{k}v||_{E_{i}^{k}}^{2}\label{PositiveEstimate1}.
\end{align}
For the second term of \eqref{PositiveEstimate1}, by the Cauchy-Schwarz inequality, we get
\begin{align}
&\ \ \ \ ||T_{0,i}^{k}v||_{E_{i}^{k}}^{2}=(T_{0,i}^{k}v,v)_{E_{i}^{k}}
=(Q_{s+1}^{h}T_{0,i}^{k}v,v)_{E_{i}^{k}}
\leq ||Q_{s+1}^{h}T_{0,i}^{k}v||_{E_{i}^{k}}||v||_{E_{i}^{k}}\notag\\
&\leq ||Q_{s+1}^{h}T_{0,i}^{k}v||_{a}||v||_{E_{i}^{k}}
\leq ||T_{0,i}^{k}v||_{a}||v||_{E_{i}^{k}}
\leq \sqrt{\beta(\lambda_{s+1}^{H})}\ ||T_{0,i}^{k}v||_{E_{i}^{k}}||v||_{E_{i}^{k}}, \label{PositiveEstimate2}
\end{align}
which yields
\begin{equation}\label{PositiveEstimate4}
||T_{0,i}^{k}v||_{E_{i}^{k}}^{2}\leq \beta(\lambda_{s+1}^{H})||v||_{E_{i}^{k}}^{2}.
\end{equation}
For the third term of \eqref{PositiveEstimate1}, by the Cauchy-Schwarz inequality, we dedcue
\begin{align}
\sum_{l=1}^{N}||T_{l,i}^{k}v||_{E_{i}^{k}}^{2}=\sum_{l=1}^{N}(T_{l,i}^{k}v,v)_{E_{i}^{k}}=(Q_{s+1}^{h}\sum_{l=1}^{N}T_{l,i}^{k}v,v)_{E_{i}^{k}}
\leq ||Q_{s+1}^{h}\sum_{l=1}^{N}T_{l,i}^{k}v||_{E_{i}^{k}}||v||_{E_{i}^{k}}.\label{PositiveEstimate5}
\end{align}
By Lemma \ref{Lemma_strengthenedCauchySchwarzinequality}, we get
\begin{align}
||Q_{s+1}^{h}\sum_{l=1}^{N}T_{l,i}^{k}v||_{E_{i}^{k}}^{2}
\leq ||\sum_{l=1}^{N}T_{l,i}^{k}v||_{a}^{2}\leq N_{0}\sum_{l=1}^{N}||T_{l,i}^{k}v||_{a}^{2}\leq N_{0}\max_{1\leq l\leq N}\beta(\lambda_{1,l}^{h})\sum_{l=1}^{N}||T_{l,i}^{k}v||_{E_{i}^{k}}^{2},\label{PositiveEstimate6}
\end{align}
where $\lambda_{1,l}^{h}=\lambda_{min}(A^{(l)})=O(H_{l}^{-2})$.
Using \eqref{PositiveEstimate5} and \eqref{PositiveEstimate6}, we obtain
\begin{equation}\label{PositiveEstimate7}
\sum_{l=1}^{N}||T_{l,i}^{k}v||_{E_{i}^{k}}^{2}\leq N_{0}\max_{1\leq l\leq N}\beta(\lambda_{1,l}^{h})||v||_{E_{i}^{k}}^{2}.
\end{equation}
Combining \eqref{PositiveEstimate1}, \eqref{PositiveEstimate4} and \eqref{PositiveEstimate7}, we know that for any $v\in U_{s+1}^{h}$,
\begin{align*}
(G_{i}^{k}v,v)_{E_{i}^{k}}
&=||v||_{E_{i}^{k}}^{2}-\alpha_{i}^{k}||T_{0,i}^{k}v||_{E_{i}^{k}}^{2}
-\alpha_{i}^{k}\sum_{l=1}^{N}||T_{l,i}^{k}v||_{E_{i}^{k}}^{2}\\
&\geq \{1-\alpha_{i}^{k}(\beta(\lambda_{s+1}^{H})+N_{0}\max_{1\leq l\leq N}\beta(\lambda_{1,l}^{h}))\}||v||_{E_{i}^{k}}^{2}.
\end{align*}
Taking $0<\alpha_{i}^{k}<\alpha_{i,0}^{k}=\frac{1}{\beta(\lambda_{s+1}^{H})+N_{0}\max_{1\leq l\leq N}\beta(\lambda_{1,l}^{h})},$ we complete the proof of this lemma. \qed
\end{proof}
\begin{remark}\label{remark_sum}
By Lemma \ref{Lemma_strengthenedCauchySchwarzinequality}, \eqref{PositiveEstimate6} and \eqref{PositiveEstimate7}, we have
\begin{align*}
||\sum_{l=1}^{N}T_{l,i}^{k}v||_{a}^{2}\leq N_{0}\sum_{l=1}^{N}||T_{l,i}^{k}v||_{a}^{2}
\leq C\sum_{l=1}^{N}||T_{l,i}^{k}v||_{E_{i}^{k}}^{2}\leq C||v||_{E_{i}^{k}}^{2}\ \ \ \ \forall\ v\in U_{s+1}^{h}.
\end{align*}
Moreover,
\begin{align*}
||\sum_{l=1}^{N}T_{l,i}^{k}v||_{b}^{2}\leq N_{0}\sum_{l=1}^{N}||T_{l,i}^{k}v||_{b}^{2}
\leq CH^{2}\sum_{l=1}^{N}||T_{l,i}^{k}v||_{a}^{2}\leq CH^{2}||v||_{E_{i}^{k}}^{2}\ \ \ \ \forall\ v\in U_{s+1}^{h}.
\end{align*}
\end{remark}
\begin{lemma}\label{Lemma_StableDecomposition}
Under the same assumptions as in Theorem \ref{Gkv2theorem}, for any $v\in U_{s+1}^{h}$, there exist $w_{0}\in U_{s+1}^{H}$ and $w^{(l)}\in V^{(l)}\ (l=1,2,...,N),$ such that
\begin{equation}\notag
v=Q_{s+1}^{h}w_{0}+\sum_{l=1}^{N}Q_{s+1}^{h}w^{(l)},
\end{equation}
and
\begin{equation}\label{ErrorStableDecomposition}
(w_{0},w_{0})_{E_{i}^{k}}+\sum_{l=1}^{N}(w^{(l)},w^{(l)})_{E_{i}^{k}}\leq C(1+\frac{H}{\delta})(v,v)_{E_{i}^{k}},\ \ i=1,2,...,s.
\end{equation}
\end{lemma}
\begin{proof}
For any $v\in U_{s+1}^{h}$, set $w_{0}=Q_{s+1}^{H}Q^{H}v$ and $w^{(l)}=I^{h}(\theta_{l}(v-w_{0}))$, where $I^{h}:C^{0}(\bar{\Omega})\to V^{h}$ is the usual nodal interpolation operator. It is easy to check that
$$Q_{s+1}^{h}w_{0}+\sum_{l=1}^{N}Q_{s+1}^{h}w^{(l)}=Q_{s+1}^{h}w_{0}+Q_{s+1}^{h}I^{h}\{(\sum_{l=1}^{N}\theta_{l})(v-w_{0})\}=v.$$
Next, we prove \eqref{ErrorStableDecomposition}. For the coarse component, we deduce
\begin{align}
||w_{0}||_{E_{i}^{k}}^{2}\leq ||Q_{s+1}^{H}Q^{H}v||_{a}^{2}
\leq ||Q^{H}v||_{a}^{2}\leq C||v||^{2}_{a}
\leq C||v||^{2}_{E_{i}^{k}}.\label{CoarseComponent}
\end{align}
For the local fine components, by the property of the operator $I^{h}$ (see Lemma 3.9 in \cite{ToselliM}), we have
\begin{align}
&\ \ \ \ \sum_{l=1}^{N}(w^{(l)},w^{(l)})_{E_{i}^{k}}\leq \sum_{l=1}^{N}a(w^{(l)},w^{(l)})
=\sum_{l=1}^{N}|I^{h}(\theta_{l}(v-w_{0}))|_{1,\Omega_{l}^{'}}^{2}\notag\\
&\leq C\sum_{l=1}^{N}|\theta_{l}(v-w_{0})|_{1,\Omega_{l}^{'}}^{2}
\leq C\sum_{l=1}^{N}(||\nabla{(v-w_{0})}||_{b,\Omega_{l}^{'}}^{2}
+\frac{1}{\delta_{l}^{2}}||v-w_{0}||_{b,\Omega_{l,\delta_{l}}}^{2}),\label{LocalFineComponent1}
\end{align}
where $||v||^{2}_{b,\widetilde{\Omega}}=\int_{\widetilde{\Omega}}v^{2}dx,\ |v|^{2}_{1,\widetilde{\Omega}}=\int_{\widetilde{\Omega}}\nabla{v}\cdot \nabla{v}dx$ for all $\widetilde{\Omega}\subset \Omega$.
On one hand, by \eqref{CoarseComponent}, we get
\begin{align}
\sum_{l=1}^{N}||\nabla{(v-w_{0})}||_{b,\Omega_{l}^{'}}^{2}\leq C||\nabla{(v-w_{0})}||_{b}^{2}
\leq C\{||v||_{a}^{2}+||w_{0}||_{a}^{2}\}\leq C||v||_{E_{i}^{k}}^{2}.\label{LocalFineComponent2}
\end{align}
On the other hand, by Lemma \ref{Lemma_SmallOvelapping} and \eqref{LocalFineComponent2}, we obtain
\begin{align}
\sum_{l=1}^{N}\frac{1}{\delta_{l}^{2}}||v-w_{0}||_{b,\Omega_{l,\delta_{l}}}^{2}
&\leq C\sum_{l=1}^{N}\{(1+\frac{H_{l}}{\delta_{l}})|v-w_{0}|^{2}_{1,\Omega_{l}^{'}}
+\frac{1}{H_{l}\delta_{l}}||v-w_{0}||^{2}_{b,\Omega_{l}^{'}}\}\notag\\
&\leq C(1+\frac{H}{\delta})||v||_{E_{i}^{k}}^{2}+ C\sum_{l=1}^{N}\frac{1}{H_{l}\delta_{l}}||v-w_{0}||^{2}_{b,\Omega_{l}^{'}}.\label{LocalFineComponent3}
\end{align}
Furthermore,
\begin{align}
\sum_{l=1}^{N}\frac{1}{H_{l}\delta_{l}}||v-w_{0}||^{2}_{b,\Omega_{l}^{'}}
\leq \frac{1}{\min_{1\leq l\leq N}{(H_{l}\delta_{l}})}\sum_{l=1}^{N}||v-w_{0}||^{2}_{b,\Omega_{l}^{'}}
\leq \frac{C}{\min_{1\leq l\leq N}{(H_{l}\delta_{l}})}||v-w_{0}||^{2}_{b}.\label{J2l_TheSecond}
\end{align}
By \eqref{Equation_QsHQHQs1hvh} and the Poincar\'e inequality, we get
\begin{align*}
||v-w_{0}||_{b}
&\leq ||v-Q^{H}v||_{b}+||Q_{s}^{H}Q^{H}v||_{b}
\leq CH||v||_{a}+CH^{2}||v||_{b}\leq CH||v||_{E_{i}^{k}},
\end{align*}
which, together with \eqref{LocalFineComponent3}, \eqref{J2l_TheSecond} and the fact that $\mathcal{J}_{H}$ is quasi-uniform, yields
\begin{align}
\sum_{l=1}^{N}\frac{1}{\delta_{l}^{2}}||v-w_{0}||_{b,\Omega_{l,\delta_{l}}}^{2}
\leq C(1+\frac{H}{\delta})||v||_{E_{i}^{k}}^{2}.\label{LocalFineComponent4}
\end{align}
Combining \eqref{CoarseComponent},\eqref{LocalFineComponent1},\eqref{LocalFineComponent2} and \eqref{LocalFineComponent4} together, we complete the proof of \eqref{ErrorStableDecomposition}.  \qed
\end{proof}

\begin{remark}
After carefully checking our proof of Lemma \ref{Lemma_StableDecomposition}, we know that the argument of the proof may be extended to the case of the fourth order symmetric elliptic operator. For the case of small overlap, we have the similar result as Lemma \ref{Lemma_StableDecomposition}, and we only need to modify \eqref{ErrorStableDecomposition}
to
\begin{equation}\notag
(w_{0},w_{0})_{E_{i}^{k}}+\sum_{l=1}^{N}(w^{(l)},w^{(l)})_{E_{i}^{k}}\leq C(1+\frac{H^{3}}{\delta^{3}})(v,v)_{E_{i}^{k}},\ \ i=1,2,...,s.
\end{equation}
\end{remark}
\par\noindent{\bf Proof of Theorem \ref{Gkv2theorem}}:\
For any $v\in U_{s+1}^{h}$, by Lemma \ref{Lemma_StableDecomposition}, there exist $w_{0}\in U_{s+1}^{H}$ and $w_{l}\in V^{(l)}$ such that
\begin{equation}\label{stabledecomposition}
v=Q_{s+1}^{h}w_{0}+\sum_{l=1}^{N}Q_{s+1}^{h}w_{l}\ \ \ \text{and}\ \ \sum_{l=0}^{N}||w_{l}||^{2}_{E_{i}^{k}}\leq C(1+\frac{H}{\delta})||v||_{E_{i}^{k}}^{2}.
\end{equation}
By \eqref{Definition_T0ik}, \eqref{Definition_Tlik}, \eqref{PositiveEstimate5}, \eqref{stabledecomposition} and the Cauchy-Schwarz inequality, we may obtain
\begin{equation}
(v,v)_{E_{i}^{k}}\leq C(1+\frac{H}{\delta})(Q_{s+1}^{h}\sum_{l=0}^{N}T_{l,i}^{k}v,v)_{E_{i}^{k}}.
\end{equation}
Moreover,
\begin{align*}
(G_{i}^{k}v,v)_{E_{i}^{k}}&=(v,v)_{E_{i}^{k}}-
\alpha_{i}^{k}(Q_{s+1}^{h}\sum_{l=0}^{N}T_{l,i}^{k}v,v)_{E_{i}^{k}}\\
&\leq \{1-\frac{\alpha_{i}^{k}}{C(1+\frac{H}{\delta})}\}(v,v)_{E_{i}^{k}}
\leq (1-C\frac{\delta}{H})(v,v)_{E_{i}^{k}}.
\end{align*}
By Lemma \ref{Lemma_SymmetricPositive}, we obtain
\begin{align}\notag
||G_{i}^{k}v||_{E_{i}^{k}}\leq ||G_{i}^{k}||_{E_{i}^{k}}||v||_{E_{i}^{k}}=\sup_{ v\ne0,v\in U_{s+1}^{h}}\frac{(G_{i}^{k}v,v)_{E_{i}^{k}}}{(v,v)_{E_{i}^{k}}}||v||_{E_{i}^{k}}\leq (1-C\frac{\delta}{H})||v||_{E_{i}^{k}},
\end{align}
which completes the proof of this theorem. \qed

\subsubsection{Estimate of the additional error term $I_{2,i}^{k}$}

\par In this subsection, we give an estimate for the additional error term $I_{2,i}^{k}$.

\par For convenience, denote by $R_{i,s+1}^{k}:=Q_{s+1}^{h}(B_{0,i}^{k})^{-1}Q_{s+1}^{H}Q^{H}$ and $\widetilde{R}_{i,s+1}^{k}:=(B_{0,i}^{k})^{-1}Q_{s+1}^{H}Q^{H}$. Similarly, denote by
\[S_{i,s+1}^{k}:=Q_{s+1}^{h}\sum_{l=1}^{N}(B_{l,i}^{k})^{-1}Q^{(l)}
\ \text{and}\ \widetilde{S}_{i,s+1}^{k}:=\sum_{l=1}^{N}(B_{l,i}^{k})^{-1}Q^{(l)}.\]
Hence, the additional error term $I_{2,i}^{k}$ defined in \eqref{SumI1ikI2ik} may be written as
\begin{align}
I_{2,i}^{k}=-\alpha_{i}^{k}(R_{i,s+1}^{k}+S_{i,s+1}^{k})(\lambda_{i}^{k}-A^{h})Q_{s}^{h}u_{i}^{k}
+\alpha_{i}^{k}Q_{s+1}^{h}Q_{U^{k}}(B_{i}^{k})^{-1}r_{i}^{k}. \label{Equation_SecondIi2k}
\end{align}

\begin{theorem}\label{Theorem_DDMassociate}
If Assumption \ref{Assumption1} and Assumption \ref{Assumption2} hold, then
\begin{equation}\notag
||I_{2,i}^{k}||_{E_{i}^{k}}\leq CH||e_{i,s+1}^{k}||_{E_{i}^{k}}+CH^{2}||g_{i}^{k}||_{a}.
\end{equation}
\end{theorem}
\begin{proof}
Firstly, we estimate the first term of $I_{2,i}^{k}$ in \eqref{Equation_SecondIi2k}. For any $w\in U_{s}^{h}$, by \eqref{Equation_Qs1HQHQshvh},  \eqref{PositiveEstimate4} and the Cauchy-Schwarz inequality , we get
\begin{align*}
&\ \ \ \ ||R_{i,s+1}^{k}w||_{E_{i}^{k}}^{2}
=b(Q_{s+1}^{h}(B_{0,i}^{k})^{-1}Q_{s+1}^{H}Q^{H}w,(A^{h}-\lambda_{i}^{k})R_{i,s+1}^{k}w)\\
&=b((B_{0,i}^{k})^{-1}Q_{s+1}^{H}Q^{H}w,(A^{h}-\lambda_{i}^{k})R_{i,s+1}^{k}w)
=b(Q_{s+1}^{H}Q^{H}w,T_{0,i}^{k}R_{i,s+1}^{k}w)\\
&\leq ||Q_{s+1}^{H}Q^{H}w||_{b}||T_{0,i}^{k}R_{i,s+1}^{k}w||_{b}
\leq CH^{2}||w||_{b}||T_{0,i}^{k}R_{i,s+1}^{k}w||_{E_{i}^{k}}
\leq CH^{2}||w||_{b}||R_{i,s+1}^{k}w||_{E_{i}^{k}},
\end{align*}
which means that $||R_{i,s+1}^{k}w||_{E_{i}^{k}}\leq CH^{2}||w||_{b}.$
In particular, we take $w=(\lambda_{i}^{k}-A^{h})Q_{s}^{h}u_{i}^{k}$. By Lemma \ref{Lemma_Relationu1e2}, we know
\begin{align}
||R_{i,s+1}^{k}(\lambda_{i}^{k}-A^{h})Q_{s}^{h}u_{i}^{k}||_{E_{i}^{k}}
\leq CH^{2}||(\lambda_{i}^{k}-A^{h})Q_{s}^{h}u_{i}^{k}||_{b}
\leq CH^{2}||g_{i}^{k}||_{a}. \label{EstimateforI11}
\end{align}
For any $w\in U_{s}^{h}$, by the Poincar\'e inequality in $V^{(l)}$ and \eqref{PositiveEstimate7}, we obtain
\begin{align*}
||S_{i,s+1}^{k}w||_{E_{i}^{k}}^{2}
&=\sum_{l=1}^{N}b((B_{l,i}^{k})^{-1}Q^{(l)}w,(A^{h}-\lambda^{k})S_{i,s+1}^{k}w)
=\sum_{l=1}^{N}b(Q^{(l)}w,T_{l,i}^{k}S_{i,s+1}^{k}w)\\
&\leq  \{\sum_{l=1}^{N}||Q^{(l)}w||_{b,\Omega_{l}^{'}}^{2}\}^{\frac{1}{2}} \{\sum_{l=1}^{N} ||T_{l,i}^{k}S_{i,s+1}^{k}w||_{b,\Omega_{l}^{'}}^{2}\}^{\frac{1}{2}}
\leq CH^{2}||w||_{a}||S_{i,s+1}^{k}w||_{E_{i}^{k}},
\end{align*}
which means that $||S_{i,s+1}^{k}w||_{E_{i}^{k}}\leq  CH^{2}||w||_{a}\leq  C\sqrt{\lambda_{s}^{h}}H^{2}||w||_{b}.$ Specially, we take $w=(\lambda_{i}^{k}-A^{h})Q_{s}^{h}u_{i}^{k}$. By Lemma \ref{Lemma_Relationu1e2}, we know
\begin{equation}\label{EstimateforI12}
||S_{i,s+1}^{k}(\lambda_{i}^{k}-A^{h})Q_{s}^{h}u_{i}^{k}||_{E_{i}^{k}}\leq CH^{2}||(\lambda_{i}^{k}-A^{h})Q_{s}^{h}u_{i}^{k}||_{b}
\leq CH^{2}||g_{i}^{k}||_{a},
\end{equation}
which, together with \eqref{EstimateforI11}, yields
\begin{equation}\label{I2ik_FirstTerm}
||-\alpha_{i}^{k}(R_{i,s+1}^{k}+S_{i,s+1}^{k})(\lambda_{i}^{k}-A^{h})Q_{s}^{h}u_{i}^{k}||_{E_{i}^{k}}\leq CH^{2}||g_{i}^{k}||_{a}.
\end{equation}
\par Secondly, we estimate the second term of $I_{2,i}^{k}$ in \eqref{Equation_SecondIi2k}. We divide it into three terms:
\begin{align}
Q_{s+1}^{h}Q_{U^{k}}(B_{i}^{k})^{-1}r_{i}^{k}
&=Q_{s+1}^{h}Q_{U^{k}}\widetilde{R}_{i,s+1}^{k}(\lambda_{i}^{k}-A^{h})Q_{s}^{h}u_{i}^{k} +Q_{s+1}^{h}Q_{U^{k}}\widetilde{S}_{i,s+1}^{k}(\lambda_{i}^{k}-A^{h})Q_{s}^{h}u_{i}^{k}\notag\\
&\ \ \ \ +Q_{s+1}^{h}Q_{U^{k}}(B_{i}^{k})^{-1}(A^{h}-\lambda_{i}^{k})e_{i,s+1}^{k}=: L_{1}+L_{2}+L_{3}.\label{Estimate_L1L2L3}
\end{align}
We estimate \eqref{Estimate_L1L2L3} one by one. Denote by $r_{i,s}^{k}:=(\lambda_{i}^{k}-A^{h})Q_{s}^{h}u_{i}^{k}$ and $ r_{i,s+1}^{k}:=(A^{h}-\lambda_{i}^{k})e_{i,s+1}^{k}$, and we know $r_{i}^{k}=r_{i,s}^{k}+r_{i,s+1}^{k}$.
 For $L_{1}$ in \eqref{Estimate_L1L2L3}, by Lemma \ref{Lemma_NormQs+1h}, Corollary \ref{Lemma_theta_ak} and Lemma \ref{Lemma_Theta_EnergyL2}, we have
\begin{align}
||L_{1}||_{E_{i}^{k}}
&\leq ||Q_{s+1}^{h}Q_{U^{k}}Q_{s+1}^{h}\widetilde{R}_{i,s+1}^{k}r_{i,s}^{k}||_{a}
+||Q_{s+1}^{h}Q_{U^{k}}Q_{s}^{h}\widetilde{R}_{i,s+1}^{k}r_{i,s}^{k}||_{a}\notag\\
&= ||Q_{s+1}^{h}P_{U^{k}}Q_{U^{k}}Q_{s+1}^{h}\widetilde{R}_{i,s+1}^{k}r_{i,s}^{k}||_{a}
+||Q_{s+1}^{h}P_{U^{k}}Q_{U^{k}}Q_{s}^{h}\widetilde{R}_{i,s+1}^{k}r_{i,s}^{k}||_{a}\notag\\
&\leq CH||R_{i,s+1}^{k}r_{i,s}^{k}||_{a}+
CH||Q_{s}^{h}\widetilde{R}_{i,s+1}^{k}r_{i,s}^{k}||_{a}. \label{Estimate_L11}
\end{align}
Note that
\begin{align}
||Q_{s}^{h}\widetilde{R}_{i,s+1}^{k}r_{i,s}^{k}||_{a}
&\leq \sqrt{\lambda_{s}^{h}}||Q_{s}^{h}(B_{0,i}^{k})^{-1}Q_{s+1}^{H}Q^{H}r_{i,s}^{k}||_{b}
\leq CH^{2}||(B_{0,i}^{k})^{-1}Q_{s+1}^{H}Q^{H}r_{i,s}^{k}||_{b}\notag\\
&\leq \frac{CH^{2}}{\lambda_{s+1}^{H}-\lambda_{i}^{k}}||Q_{s+1}^{H}Q^{H}r_{i,s}^{k}||_{b}
\leq CH^{4}||(\lambda_{i}^{k}-A^{h})Q_{s}^{h}u_{i}^{k}||_{b}
\leq CH^{4}||g_{i}^{k}||_{a}. \label{Estimate_L12}
\end{align}
By \eqref{EstimateforI11},\ \eqref{Estimate_L11} and \eqref{Estimate_L12}, we obtain
\begin{equation}\label{Estimate_L1}
||L_{1}||_{E_{i}^{k}}\leq CH^{3}||g_{i}^{k}||_{a}+CH^{5}||g_{i}^{k}||_{a}\leq CH^{3}||g_{i}^{k}||_{a} .
\end{equation}
For $L_{2}$ in \eqref{Estimate_L1L2L3}, we deduce
\begin{align}
||L_{2}||_{E_{i}^{k}}
&\leq ||Q_{s+1}^{h}Q_{U^{k}}Q_{s+1}^{h}\widetilde{S}_{i,s+1}^{k}r_{i,s}^{k}||_{a}
+||Q_{s+1}^{h}Q_{U^{k}}Q_{s}^{h}\widetilde{S}_{i,s+1}^{k}r_{i,s}^{k}||_{a}\notag\\
&=||Q_{s+1}^{h}P_{U^{k}}Q_{U^{k}}Q_{s+1}^{h}\widetilde{S}_{i,s+1}^{k}r_{i,s}^{k}||_{a}
+||Q_{s+1}^{h}P_{U^{k}}Q_{U^{k}}Q_{s}^{h}\widetilde{S}_{i,s+1}^{k}r_{i,s}^{k}||_{a}\notag\\
&\leq CH||S_{i,s+1}^{k}r_{i,s}^{k}||_{a}
+CH||Q_{s}^{h}\widetilde{S}_{i,s+1}^{k}r_{i,s}^{k}||_{a}.\label{Estimate_L21}
\end{align}
In addition, by Lemma \ref{Lemma_strengthenedCauchySchwarzinequality}, Lemma \ref{Lemma_Relationu1e2} and the Poincar\'e inequality in $V^{(l)}$, we get
\begin{align}
||Q_{s}^{h}\widetilde{S}_{i,s+1}^{k}r_{i,s}^{k}||_{a}^{2}
&\leq \lambda_{s}^{h}||\widetilde{S}_{i,s+1}^{k}r_{i,s}^{k}||_{b}^{2}
\leq CN_{0}\sum_{l=1}^{N}||(B_{l,i}^{k})^{-1}Q^{(l)}r_{i,s}^{k}||^{2}_{b}
\leq  CH^{4}\sum_{l=1}^{N}||Q^{(l)}r_{i,s}^{k}||^{2}_{b}\notag\\
&\leq  CH^{6}\sum_{l=1}^{N}||Q^{(l)}r_{i,s}^{k}||^{2}_{a}
\leq CH^{6}||r_{i,s}^{k}||_{a}^{2}
\leq C\lambda_{s}^{h}H^{6}||r_{i,s}^{k}||_{b}^{2}
\leq CH^{6}||g_{i}^{k}||_{a}^{2}.\label{Estimate_L22}
\end{align}
Combining \eqref{EstimateforI12}, \eqref{Estimate_L21} and \eqref{Estimate_L22}, we obtain
\begin{equation}\label{Estimate_L2}
||L_{2}||_{E_{i}^{k}}\leq CH^{3}||g_{i}^{k}||_{a}+CH^{4}||g_{i}^{k}||_{a}\leq CH^{3}||g_{i}^{k}||_{a}.
\end{equation}
For $L_{3}$ in \eqref{Estimate_L1L2L3}, by Lemma \ref{Lemma_QUsQs1}, Remark \ref{remark_sum} and \eqref{PositiveEstimate4}, we have
\begin{align}
||L_{3}||_{E_{i}^{k}}
&\leq ||Q_{s+1}^{h}Q_{U^{k}}Q_{s+1}^{h}(B_{i}^{k})^{-1}r_{i,s+1}^{k}||_{a}+
||Q_{s+1}^{h}Q_{U^{k}}Q_{s}^{h}(B_{i}^{k})^{-1}r_{i,s+1}^{k}||_{a}\notag\\
&=||Q_{s+1}^{h}P_{U^{k}}Q_{U^{k}}Q_{s+1}^{h}(B_{i}^{k})^{-1}r_{i,s+1}^{k}||_{a}+
||Q_{s+1}^{h}P_{U^{k}}Q_{U^{k}}Q_{s}^{h}(B_{i}^{k})^{-1}r_{i,s+1}^{k}||_{a}\notag\\
&\leq CH||(B_{i}^{k})^{-1}r_{i,s+1}^{k}||_{a}+CH||Q_{s}^{h}(B_{i}^{k})^{-1}r_{i,s+1}^{k}||_{a}\notag\\
&\leq CH||T_{0,i}^{k}e_{i,s+1}^{k}||_{a}+CH||\sum_{l=1}^{N}T_{l,i}^{k}e_{i,s+1}^{k}||_{a}+ CH||Q_{s}^{h}T_{0,i}^{k}e_{i,s+1}^{k}||_{b}\notag\\
&\ \ \ \ +CH||Q_{s}^{h}\sum_{l=1}^{N}T_{l,i}^{k}e_{i,s+1}^{k}||_{b}
\leq  CH||e_{i,s+1}^{k}||_{E_{i}^{k}},\label{align_L3}
\end{align}
which, together with \eqref{Estimate_L1L2L3}, \eqref{Estimate_L1}, \eqref{Estimate_L2}, yields
\begin{equation}\label{I2ik_SecondTerm}
||\alpha_{i}^{k}Q_{s+1}^{h}Q_{U^{k}}(B_{i}^{k})^{-1}r_{i}^{k}||_{E_{i}^{k}}\leq CH||e_{i,s+1}^{k}||_{E_{i}^{k}}+CH^{3}||g_{i}^{k}||_{a}.
\end{equation}
\par Finally, combining \eqref{Equation_SecondIi2k}, \eqref{I2ik_FirstTerm} and \eqref{I2ik_SecondTerm}, we may complete the proof of this theorem.\qed
\end{proof}
\begin{theorem}\label{Theorem_DDMEstimate}
If Assumption \ref{Assumption1} and Assumption \ref{Assumption2} hold, we have
\begin{equation}\notag
||\widetilde{e}_{i,s+1}^{k+1}||_{E_{i}^{k}}\leq c_{0}(H)(1-C\frac{\delta}{H})||e_{i,s+1}^{k}||_{E_{i}^{k}}+CH^{2}||g_{i}^{k}||_{a},\ \ \ i=1,2,...,s,
\end{equation}
where $H$-dependent constant $c_{0}(H)\ (=1+\frac{CH}{1-C\frac{\delta}{H}})$ decreases monotonically to $1$, as $H \to 0$.
\end{theorem}
\begin{proof}
By \eqref{SumI1ikI2ik}, Theorem \ref{Gkv2theorem} and Theorem \ref{Theorem_DDMassociate}, we have
\begin{align*}
&\ \ \ \ ||\widetilde{e}_{i,s+1}^{k+1}||_{E_{i}^{k}}
\leq ||I_{1,i}^{k}||_{E_{i}^{k}}+||I_{2,i}^{k}||_{E_{i}^{k}}
=||G_{i}^{k}e_{i,s+1}^{k}||_{E_{i}^{k}}+||I_{2,i}^{k}||_{E_{i}^{k}}\\
&\leq (1-C\frac{\delta}{H})||e_{i,s+1}^{k}||_{E_{i}^{k}}+CH||e_{i,s+1}^{k}||_{E_{i}^{k}}+CH^{2}||g_{i}^{k}||_{a}\\
&\leq \{(1-C\frac{\delta}{H})+CH\}||e_{i,s+1}^{k}||_{E_{i}^{k}}+CH^{2}||g_{i}^{k}||_{a}\\
&= c_{0}(H)(1-C\frac{\delta}{H})||e_{i,s+1}^{k}||_{E_{i}^{k}}+CH^{2}||g_{i}^{k}||_{a},
\end{align*}
where $c_{0}(H)\ (=1+\frac{CH}{1-C\frac{\delta}{H}})$ decreases monotonically to $1$, as $H \to 0$. \qed
\end{proof}
\begin{remark}\label{remark_QusrikEstimate}
By Remark \ref{remark_sum}, \eqref{Estimate_L12}, \eqref{Estimate_L22}, \eqref{align_L3} and the facts that
$(B_{i}^{k})^{-1}=\widetilde{R}_{i,s+1}^{k}+\widetilde{S}_{i,s+1}^{k}$ and $r_{i}^{k}=r_{i,s}^{k}+r_{i,s+1}^{k}$, we have
\begin{equation}\notag
||Q_{s}^{h}(B_{i}^{k})^{-1}r_{i}^{k}||_{b}\leq C||Q_{s}^{h}(B_{i}^{k})^{-1}r_{i}^{k}||_{a}\leq CH||e_{i,s+1}^{k}||_{E_{i}^{k}}+CH^{3}||g_{i}^{k}||_{a}.
\end{equation}
Similarly, by Remark \ref{remark_sum}, \eqref{PositiveEstimate4} and \eqref{I2ik_FirstTerm}, we get
\begin{equation}\notag
||Q_{s+1}^{h}(B_{i}^{k})^{-1}r_{i}^{k}||_{b}\leq C||Q_{s+1}^{h}(B_{i}^{k})^{-1}r_{i}^{k}||_{a}\leq C||e_{i,s+1}^{k}||_{E_{i}^{k}}+CH^{2}||g_{i}^{k}||_{a}.
\end{equation}
\end{remark}
\subsection{The proof of the main result}
\par In this subsection, based on Theorem \ref{Theorem_DDMEstimate} in previous subsection, we first give an estimate for
\begin{equation}\notag
\sum_{i=1}^{s}(Rq(\widetilde{u}_{i}^{k+1})-Rq(Q_{s}^{h}u_{i}^{k})),
\end{equation}
and then present a rigorous proof of Theorem \ref{MainResult}.
\begin{lemma}\label{Lemma_ConvergentRate}
Let Assumption \ref{Assumption1} and Assumption \ref{Assumption2} hold. It holds that
\begin{equation}\label{Equation_mainresult1}
\sum_{i=1}^{s}(Rq(\widetilde{u}_{i}^{k+1})-Rq(Q_{s}^{h}u_{i}^{k}))
\leq  \gamma_{0}\sum_{i=1}^{s}(\lambda_{i}^{k}-Rq(Q_{s}^{h}u_{i}^{k}))+CH\sum_{i=1}^{s}(\lambda_{i}^{k}-\lambda_{i}^{h}),
\end{equation}
where $\gamma_{0}=:(1-C\frac{\delta}{H})^{2}+CH.$
\end{lemma}
\begin{proof}
Firstly, by the fact that $||\widetilde{u}_{i}^{k+1}||_{b}^{2}=||\widetilde{e}_{i,s+1}^{k+1}||_{b}^{2}+||Q_{s}^{h}\widetilde{u}_{i}^{k+1}||_{b}^{2}$, we deduce
\begin{align}
||\widetilde{e}_{i,s+1}^{k+1}||_{E_{i}^{k}}^{2}
&=b((A^{h}-\lambda_{i}^{k})\widetilde{u}_{i}^{k+1},\widetilde{u}_{i}^{k+1})-
b((A^{h}-\lambda_{i}^{k})Q_{s}^{h}\widetilde{u}_{i}^{k+1},Q_{s}^{h}\widetilde{u}_{i}^{k+1})\notag\\
&=(Rq(\widetilde{u}_{i}^{k+1})-Rq(Q_{s}^{h}u_{i}^{k}))||\widetilde{u}_{i}^{k+1}||_{b}^{2}
+(Rq(Q_{s}^{h}u_{i}^{k})-\lambda_{i}^{k})||\widetilde{u}_{i}^{k+1}||_{b}^{2}\notag\\
&\ \ \ \ +(\lambda_{i}^{k}-Rq(Q_{s}^{h}\widetilde{u}_{i}^{k+1}))||Q_{s}^{h}\widetilde{u}_{i}^{k+1}||_{b}^{2}\notag\\
&=(Rq(\widetilde{u}_{i}^{k+1})-Rq(Q_{s}^{h}u_{i}^{k}))||\widetilde{u}_{i}^{k+1}||_{b}^{2}
+(Rq(Q_{s}^{h}u_{i}^{k})-\lambda_{i}^{k})||\widetilde{e}_{i,s+1}^{k+1}||_{b}^{2}\notag\\
&\ \ \ \ +(Rq(Q_{s}^{h}u_{i}^{k})-Rq(Q_{s}^{h}\widetilde{u}_{i}^{k+1}))||Q_{s}^{h}\widetilde{u}_{i}^{k+1}||_{b}^{2},\label{TotalSumEstimate1}
\end{align}
which yields
\begin{align}
(Rq(\widetilde{u}_{i}^{k+1})-Rq(Q_{s}^{h}u_{i}^{k}))||\widetilde{u}_{i}^{k+1}||_{b}^{2}
&=\{||\widetilde{e}_{i,s+1}^{k+1}||_{E_{i}^{k}}^{2}
+(\lambda_{i}^{k}-Rq(Q_{s}^{h}u_{i}^{k}))||\widetilde{e}_{i,s+1}^{k+1}||_{b}^{2}\}\notag\\
&\ \ \ +(Rq(Q_{s}^{h}\widetilde{u}_{i}^{k+1})-Rq(Q_{s}^{h}u_{i}^{k}))||Q_{s}^{h}\widetilde{u}_{i}^{k+1}||_{b}^{2}
=:J_{1}+J_{2}.\label{TotalSumEstimate2}
\end{align}
\par Secondly, we estimate $J_{1}$ and $J_{2}$ in \eqref{TotalSumEstimate2} one by one. For $J_{1}$, by Lemma \ref{Lemma_Relationu1e2} and Theorem \ref{Theorem_DDMEstimate}, we get
\begin{align}
J_{1}
&\leq ||\widetilde{e}_{i,s+1}^{k+1}||_{E_{i}^{k}}^{2}+(\lambda_{i}^{k}-Rq(Q_{s}^{h}u_{i}^{k}))
||\widetilde{e}_{i,s+1}^{k+1}||_{E_{i}^{k}}^{2}\notag\\
&\leq \{ ((1-C\frac{\delta}{H})+CH)||e_{i,s+1}^{k}||_{E_{i}^{k}}+CH^{2}||g_{i}^{k}||_{a}\}^{2}\notag\\
&\ \ \ \ +CH^{2}\{C||e_{i,s+1}^{k}||^{2}_{E_{i}^{k}}+CH^{4}||g_{i}^{k}||^{2}_{a}\}\notag\\
&\leq \gamma_{0}||e_{i,s+1}^{k}||^{2}_{E_{i}^{k}}+CH^{2}||g_{i}^{k}||^{2}_{a}
\leq \gamma_{0}(\lambda_{i}^{k}-Rq(Q_{U_{s}^{h}}u_{i}^{k}))+CH^{2}||g_{i}^{k}||^{2}_{a}.\label{TotalSumEstimate3}
\end{align}
For convenience, denote by $w_{i}^{k}:=Q_{s}^{h}Q_{\perp}^{k}(B_{i}^{k})^{-1}r_{i}^{k}$. For $J_{2}$ in \eqref{TotalSumEstimate2},
by \eqref{uik1}, we deduce
$$b(Q_{s}^{h}\widetilde{u}_{i}^{k+1},Q_{s}^{h}\widetilde{u}_{i}^{k+1})
=b(Q_{s}^{h}u_{i}^{k},Q_{s}^{h}u_{i}^{k})+2\alpha_{i}^{k}b(Q_{s}^{h}u_{i}^{k},w_{i}^{k})+(\alpha_{i}^{k})^{2}b(w_{i}^{k},w_{i}^{k}),$$
and
$$a(Q_{s}^{h}\widetilde{u}_{i}^{k+1},Q_{s}^{h}\widetilde{u}_{i}^{k+1})
=a(Q_{s}^{h}u_{i}^{k},Q_{s}^{h}u_{i}^{k})+2\alpha_{i}^{k}a(Q_{s}^{h}u_{i}^{k},w_{i}^{k})+(\alpha_{i}^{k})^{2}a(w_{i}^{k},w_{i}^{k}),$$
which yields
\begin{align}
J_{2}&=a(Q_{s}^{h}\widetilde{u}_{i}^{k+1},Q_{s}^{h}\widetilde{u}_{i}^{k+1})
-Rq(Q_{s}^{h}u_{i}^{k})||Q_{s}^{h}\widetilde{u}_{i}^{k+1}||_{b}^{2}\notag\\
&=2\alpha_{i}^{k}b((A^{h}-\lambda_{i}^{k})Q_{s}^{h}u_{i}^{k},w_{i}^{k}) +2\alpha_{i}^{k}b((\lambda_{i}^{k}-Rq(Q_{s}^{h}u_{i}^{k}))Q_{s}^{h}u_{i}^{k},
w_{i}^{k})\notag\\
&\ \ \ \ +(\alpha_{i}^{k})^{2}b((A^{h}-Rq(Q_{s}^{h}u_{i}^{k}))w_{i}^{k},
w_{i}^{k})
=:J_{2,1}+J_{2,2}+J_{2,3}.\label{OuterEstimate1}
\end{align}
\par For $J_{2,1}$ in \eqref{OuterEstimate1}, by the Cauchy-Schwarz inequality, Lemma \ref{Lemma_theta_bk}, Lemma \ref{Lemma_Theta_EnergyL2}, Lemma \ref{Lemma_Relationu1e2} and Remark \ref{remark_QusrikEstimate}, we get
\begin{align}
J_{2,1}
&\leq C||(A^{h}-\lambda_{i}^{k})Q_{s}^{h}u_{i}^{k}||_{b}||w_{i}^{k}||_{b}
\leq C\theta_{b}^{k}||g_{i}^{k}||_{a}||(B_{i}^{k})^{-1}r_{i}^{k}||_{b}\notag\\
&\leq CH||g_{i}^{k}||_{a}\{C||e_{i,s+1}^{k}||_{E_{i}^{k}}+CH^{2}||g_{i}^{k}||_{a}\}
\leq CH(\lambda_{i}^{k}-Rq(Q_{s}^{h}u_{i}^{k}))+CH||g_{i}^{k}||_{a}^{2}. \label{OuterEstimate2}
\end{align}
Similarly, for $J_{2,2}$ in \eqref{OuterEstimate1}, we have
\begin{align}
J_{2,2}&\leq C(\lambda_{i}^{k}-Rq(Q_{s}^{h}u_{i}^{k}))||w_{i}^{k}||_{b}
\leq C\theta_{b}^{k}(\lambda_{i}^{k}-Rq(Q_{s}^{h}u_{i}^{k}))||(B_{i}^{k})^{-1}r_{i}^{k}||_{b}\notag\\
&\leq CH^{2}(\lambda_{i}^{k}-Rq(Q_{s}^{h}u_{i}^{k}))+CH^{4}||g_{i}^{k}||_{a}^{2}.\label{OuterEstimate3}
\end{align}
For $J_{2,3}$ in \eqref{OuterEstimate1},  we deduce
\begin{align}
J_{2,3}&=(\alpha_{i}^{k})^{2}\{||w_{i}^{k}||_{a}^{2}+
(\lambda_{i}^{k}-Rq(Q_{s}^{h}u_{i}^{k}))||w_{i}^{k}||_{b}^{2}-\lambda_{i}^{k}||w_{i}^{k}||_{b}^{2}\}\notag\\
&\leq C\{\lambda_{s}^{h}+\lambda_{i}^{k}-Rq(Q_{s}^{h}u_{i}^{k})\}||w_{i}^{k}||_{b}^{2}
\leq C(\theta_{b}^{k})^{2}||(B_{i}^{k})^{-1}r_{i}^{k}||_{b}^{2}\notag\\
&\leq CH^{2}(\lambda_{i}^{k}-Rq(Q_{s}^{h}u_{i}^{k}))+CH^{6}||g_{i}^{k}||_{a}^{2},\notag
\end{align}
which, together with \eqref{OuterEstimate1} \eqref{OuterEstimate2} \eqref{OuterEstimate3}, yields
\begin{align}\label{OuterEstimate4}
J_{2}&=(Rq(Q_{s}^{h}\widetilde{u}_{i}^{k+1})-Rq(Q_{s}^{h}u_{i}^{k}))||Q_{s}^{h}\widetilde{u}_{i}^{k+1}||_{b}^{2}
\leq CH(\lambda_{i}^{k}-Rq(Q_{s}^{h}u_{i}^{k}))+CH||g_{i}^{k}||_{a}^{2}.
\end{align}
\par Finally, combining \eqref{TotalSumEstimate2}, \eqref{TotalSumEstimate3} and \eqref{OuterEstimate4},  we have
\begin{align}
(Rq(\widetilde{u}_{i}^{k+1})-Rq(Q_{s}^{h}u_{i}^{k}))||\widetilde{u}_{i}^{k+1}||_{b}^{2}
&\leq \{\gamma_{0}(\lambda_{i}^{k}-Rq(Q_{s}^{h}u_{i}^{k}))+CH^{2}||g_{i}^{k}||^{2}_{a}\}\notag\\
&\ \ \ \ +\{CH(\lambda_{i}^{k}-Rq(Q_{s}^{h}u_{i}^{k}))+CH||g_{i}^{k}||_{a}^{2}\}\notag\\
&\leq \gamma_{0}(\lambda_{i}^{k}-Rq(Q_{s}^{h}u_{i}^{k}))+CH||g_{i}^{k}||^{2}_{a}.\notag
\end{align}
Since $\frac{1}{||\widetilde{u}_{i}^{k+1}||_{b}}\leq \frac{1}{||u_{i}^{k}||_{b}}=1$, we obtain
\begin{equation}\label{Equation_Sum1}
Rq(\widetilde{u}_{i}^{k+1})-Rq(Q_{s}^{h}u_{i}^{k})\leq \gamma_{0}(\lambda_{i}^{k}-Rq(Q_{s}^{h}u_{i}^{k}))+CH||g_{i}^{k}||^{2}_{a}.
\end{equation}
Taking summation over $i$ in \eqref{Equation_Sum1} and using Lemma \ref{Lemma_gik}, we complete the proof of this lemma.     \qed
\end{proof}
\par We also establish an estimate for $\sum_{i=1}^{s}(\hat{\lambda}_{i}^{k+1}-Rq(\widetilde{u}_{i}^{k+1}))$, where $\hat{\lambda}_{i}^{k+1}$ is defined in \eqref{AuxiliaryEigenvalueProblem}. In order to make the proof of our main result neat, we put the proof of following lemma (Lemma \ref{Lemma_AuxiliaryProblem}) in Appendix.
\begin{lemma}\label{Lemma_AuxiliaryProblem}
Under the same assumptions as in Lemma \ref{Lemma_ConvergentRate}, it holds that
\begin{equation}\notag
\sum_{i=1}^{s}(\hat{\lambda}_{i}^{k+1}-Rq(\widetilde{u}_{i}^{k+1}))\leq CH\sum_{i=1}^{s}(\lambda_{i}^{k}-Rq(Q_{s}^{h}u_{i}^{k}))
+CH^{5}\sum_{i=1}^{s}(\lambda_{i}^{k}-\lambda_{i}^{h}).
\end{equation}
\end{lemma}

\par Now we are in a position to prove the main result of this paper.
\par\noindent{\bf Proof of Theorem \ref{MainResult}}:\ \
By Lemma \ref{Lemma_ConvergentRate} and Lemma \ref{Lemma_AuxiliaryProblem}, we get
\begin{align}
\sum_{i=1}^{s}(\hat{\lambda}_{i}^{k+1}-\lambda_{i}^{h})
&=\sum_{i=1}^{s}(\hat{\lambda}_{i}^{k+1}-Rq(\widetilde{u}_{i}^{k+1}))
+\sum_{i=1}^{s}(Rq(\widetilde{u}_{i}^{k+1})-Rq(Q_{s}^{h}u_{i}^{k}))
+\sum_{i=1}^{s}(Rq(Q_{s}^{h}u_{i}^{k})-\lambda_{i}^{h})\notag\\
&\leq \{CH\sum_{i=1}^{s}(\lambda_{i}^{k}-Rq(Q_{s}^{h}u_{i}^{k}))+CH^{5}\sum_{i=1}^{s}(\lambda_{i}^{k}-\lambda_{i}^{h})\}
+\{\gamma_{0}\sum_{i=1}^{s}(\lambda_{i}^{k}-Rq(Q_{s}^{h}u_{i}^{k}))+\notag\\
&\ \ \ \ +CH\sum_{i=1}^{s}(\lambda_{i}^{k}-\lambda_{i}^{h})\}
+\sum_{i=1}^{s}(Rq(Q_{s}^{h}u_{i}^{k})-\lambda_{i}^{h})\notag\\
&\leq \gamma_{0}\sum_{i=1}^{s}(\lambda_{i}^{k}-Rq(Q_{s}^{h}u_{i}^{k}))+
\sum_{i=1}^{s}(Rq(Q_{s}^{h}u_{i}^{k})-\lambda_{i}^{h})+CH\sum_{i=1}^{s}(\lambda_{i}^{k}-\lambda_{i}^{h}).\notag
\end{align}
Considering $Rq(Q_{s}^{h}u_{i}^{k}) \leq \lambda_{i}^{k}$ and $\lambda_{i}^{k+1}\leq \hat{\lambda}_{i}^{k+1}$, we deduce
\begin{align*}
\sum_{i=1}^{s}(\lambda_{i}^{k+1}-\lambda_{i}^{h})
&\leq \gamma_{0}\sum_{i=1}^{s}(\lambda_{i}^{k}-\lambda_{i}^{h})
+(1-\gamma_{0})\sum_{i=1}^{s}(Rq(Q_{s}^{h}u_{i}^{k})-\lambda_{i}^{h})+\\
&\ \ \ \ +CH\sum_{i=1}^{s}(\lambda_{i}^{k}-\lambda_{i}^{h})
\leq \gamma \sum_{i=1}^{s}(\lambda_{i}^{k}-\lambda_{i}^{h}),
\end{align*}
where
$\gamma=\max\{\gamma_{0},1-\gamma_{0}\}=\gamma_{0}=(1-C\frac{\delta}{H})^{2}+CH=c(H)(1-C\frac{\delta}{H})^{2}.$
Here, without loss of generality, let $\gamma_0\geq \frac{1}{2}$. The $H$-dependent constant $c(H)\ (:=1+\frac{CH}{(1-C\frac{\delta}{H})^{2}})$ decreases monotonically to 1, as $H\to$ 0. Combining Lemma \ref{Lemma_theta_bk}, Corollary \ref{Lemma_theta_ak} and \eqref{Total_Theorem}, we may prove \eqref{Hausdorffa_Theorem} and \eqref{Hausdorffb_Theorem}, which completes the proof of this theorem.   \qed
\section{Numerical experiments}
\par In this section, we present several numerical experiments in two and three dimensional eigenvalue problems to support our theoretical findings. For the stopping criterion of the proposed method, we choose the accuracy of  $\sum_{i=1}^{s}|\lambda_{i}^{k+1}-\lambda_{i}^{k}|<tol=1e^{-10}$.

\subsection{2D Laplacian eigenvalue problem}

\par In this subsection, we shall present some numerical results of 2D Laplacian eigenvalue problems in convex and L-shaped domains.
\begin{example}
We consider the Laplacian eigenvalue problem \eqref{Laplacian} in $(0,\pi)^{2}$ and use the triangle $P_{1}$-conforming finite element to compute the first $s$ eigenpairs. First, we choose an initial uniform partition $\mathcal{J}_{H}$ in $\Omega$ with the number of subdomains $N=512$, and coarse grid size $H=\frac{\sqrt{2}\pi}{2^{4}}$. We refine uniformly the grid layer by layer and fix the ratio $\frac{\delta}{H}=\frac{1}{4}$. Next, we test the optimality and scalability of our algorithm.
\end{example}

\begin{table}[H]
 \centering
 \caption{ $N=512$, $\frac{\delta}{H}=\frac{1}{4}$, $s=19$}\label{Table1}
\newcolumntype{d}{D{.}{.}{2}}
\begin{tabular}{|c|c|c|c|c|c|c|}
\hline
\multicolumn{1}{|c|}{$d.o.f.$}&\multicolumn{1}{c|}{$16129 $}&\multicolumn{1}{c|}{$65025$}
&\multicolumn{1}{c|}{$261121$}&\multicolumn{1}{c|}{$1046529$}&\multicolumn{1}{c|}{$4190209$}&\multicolumn{1}{c|}{$16769025$}\\
\hline
\multicolumn{1}{|c|}{$\lambda_{i}$}
&\multicolumn{1}{c|}{$21(it.)$}&\multicolumn{1}{c|}{$22(it.)$}&\multicolumn{1}{c|}{$22(it.)$}
&\multicolumn{1}{c|}{$22(it.)$}&\multicolumn{1}{c|}{$22(it.)$}&\multicolumn{1}{c|}{$22(it.)$}\\
\hline
$\lambda_{1}=2$    &2.00030120  &2.00007530  &2.00001882  &2.00000471 &2.00000118 &2.00000029\\  \hline
$\lambda_{2}=5$    &5.00129490  &5.00032372  &5.00008093  &5.00002023 &5.00000506 &5.00000126\\  \hline
$\lambda_{3}=5$    &5.00201852  &5.00050458  &5.00012614  &5.00003154 &5.00000788 &5.00000197\\  \hline
$\lambda_{4}=8$    &8.00481845  &8.00120474  &8.00030119  &8.00007530 &8.00001882 &8.00000471\\  \hline
$\lambda_{5}=10$   &10.00592410 &10.00148092 &10.00037022 &10.00009256&10.00002314&10.00000578\\ \hline
$\lambda_{6}=10$   &10.00592615 &10.00148105 &10.00037023 &10.00009256&10.00002314&10.00000578\\ \hline
$\lambda_{7}=13$   &13.00904908 &13.00226266 &13.00056569 &13.00014142&13.00003536&13.00000884\\ \hline
$\lambda_{8}=13$   &13.01514849 &13.00378646 &13.00094657 &13.00023664&13.00005916&13.00001479\\ \hline
$\lambda_{9}=17$   &17.01592318 &17.00397968 &17.00099485 &17.00024871&17.00006218&17.00001554\\ \hline
$\lambda_{10}=17$  &17.01631708 &17.00407809 &17.00101945 &17.00025486&17.00006371&17.00001593\\ \hline
$\lambda_{11}=18$  &18.02436417 &18.00609718 &18.00152468 &18.00038119&18.00009530&18.00002383\\ \hline
$\lambda_{12}=20$  &20.02650464 &20.00662628 &20.00165658 &20.00041414&20.00010354&20.00002588\\ \hline
$\lambda_{13}=20$  &20.02655291 &20.00662929 &20.00165677 &20.00041416&20.00010354&20.00002588\\ \hline
$\lambda_{14}=25$  &25.03383780 &25.00846626 &25.00211699 &25.00052927&25.00013232&25.00003308\\ \hline
$\lambda_{15}=25$  &25.05779711 &25.01444795 &25.00361190 &25.00090297&25.00022574&25.00005644\\ \hline
$\lambda_{16}=26$  &26.03646327 &26.00911235 &26.00227787 &26.00056945&26.00014236&26.00003559\\ \hline
$\lambda_{17}=26$  &26.03646513 &26.00911246 &26.00227788 &26.00056945&26.00014236&26.00003559\\ \hline
$\lambda_{18}=29$  &29.05122987 &29.01279949 &29.00319937 &29.00079981&29.00019995&29.00004999\\ \hline
$\lambda_{19}=29$  &29.05337468 &29.01333488 &29.00333317 &29.00083326&29.00020831&29.00005208\\ \hline
$stop.$            &8.7853e-11  &3.1262e-11  &3.5083e-11  & 3.5370e-11&3.8950e-11 &4.2572e-11 \\ \hline
\end{tabular}
\end{table}

\par Before analyzing numerical results, we first introduce some notations used in Table \ref{Table1}, which have the same meanings as following tables. We denote by $d.o.f.$ degrees of freedom, by $it.$ the number of iterations and by $stop.$ the total error between two adjacent iterative eigenvalues $\sum_{i=1}^{s}|\lambda_{i}^{k+1}-\lambda_{i}^{k}|$ when exiting the outer loop in our two-level BPJD method. Here we present the total error between two adjacent iterative eigenvalues $\sum_{i=1}^{s}|\lambda_{i}^{k+1}-\lambda_{i}^{k}|$ in order to verify our main convergence result. It is shown in Table 1 that the number of iterations of the proposed method keeps stable when $d.o.f.\to +\infty$, which illustrates that our method is optimal. It is seen from Figure \ref{optimality} that all of the curves of the total error $\sum_{i=1}^{s}|\lambda_{i}^{k+1}-\lambda_{i}^{k}|$ with different degrees of freedom coincide, which verifies that the convergence rate of the proposed method is independent of $h$. In order to test the scalability of our algorithm, we set $d.o.f.=16769025$ and the ratio $\frac{\delta}{H}=\frac{1}{4}$ to observe the relationship between the number of iterations and the number of subdomains.

\par  It is obvious to see in Table \ref{Table2} that the number of iterations decreases, as the number of subdomains increases, which shows that our algorithm is scalable. More intuitively, it is observed in Figure \ref{scalability} that curves of $\sum_{i=1}^{s}|\lambda_{i}^{k+1}-\lambda_{i}^{k}|$ with different subdomains are almost parallel, which illustrates that our algorithm has a good scalability. Although our theoretical analysis only holds for convex cases, our algorithm still works very well for nonconvex cases. We present some numerical results for the 2D Laplacian eigenvalue problem in L-shape domain.
\begin{table}[H]
 \centering
 \caption{$N=512,2048,8192$, $\frac{\delta}{H}=\frac{1}{4}$, $s=19$}\label{Table2}
\newcolumntype{d}{D{.}{.}{2}}
\begin{tabular}{|c|c|c|}
\hline
$N$&\multicolumn{1}{c|}{$d.o.f.$}&\multicolumn{1}{c|}{$it.$}\\
\hline
512&16769025&22\\
\hline
2048&16769025&18\\
\hline
8192&16769025&16\\
\hline
\end{tabular}
\end{table}

\begin{figure}[H]
  \centering
  \includegraphics[width=9cm,height=6cm]{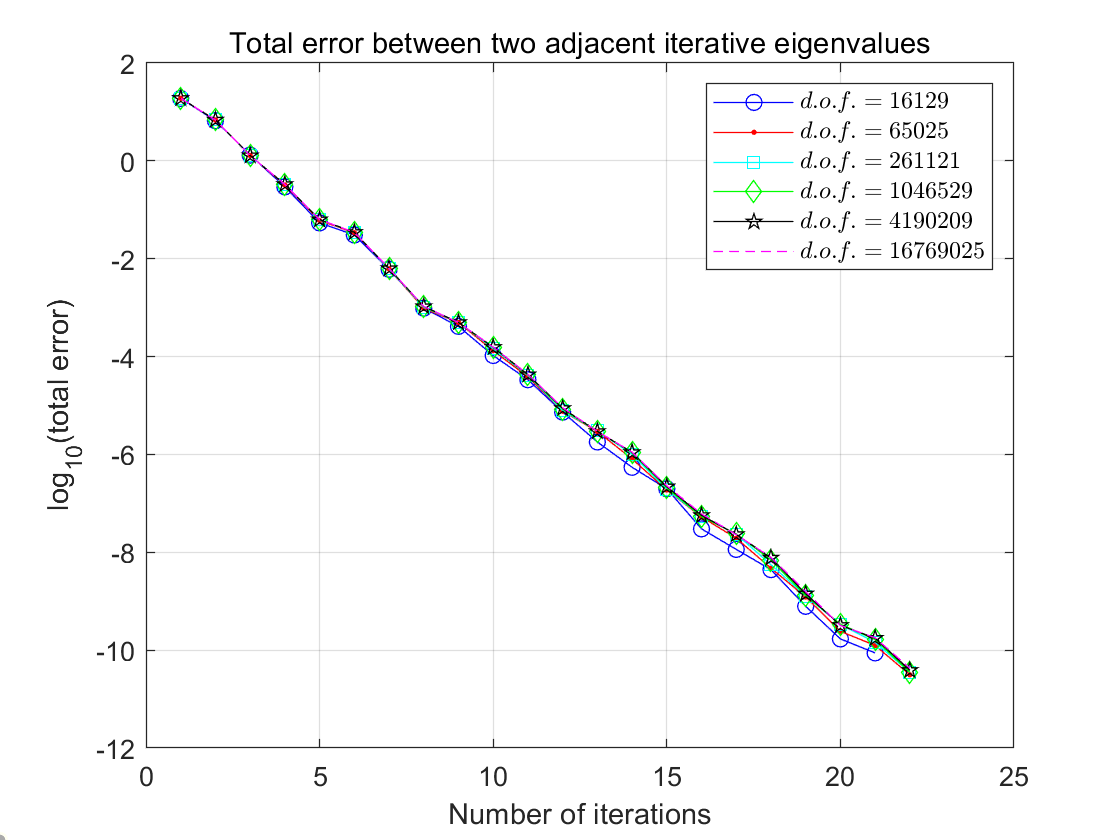}\\
  \caption{$N=512$,\ $\frac{\delta}{H}=\frac{1}{4}$,\ $s=19$}\label{optimality}
\end{figure}
\begin{figure}[H]
  \centering
  \includegraphics[width=9cm,height=6cm]{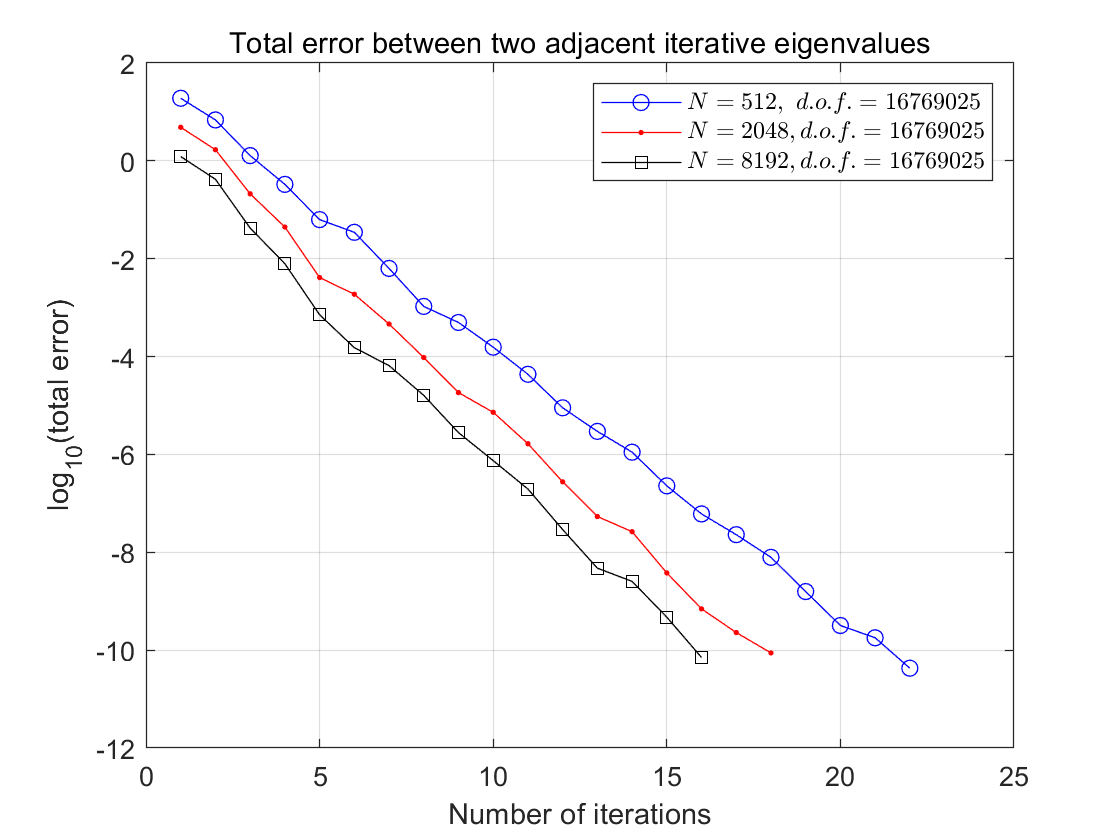}\\
  \caption{$d.o.f.=16769025$,\ $\frac{\delta}{H}=\frac{1}{4}$,\ $s=19$}\label{scalability}
\end{figure}

\begin{example}
We consider the Laplacian eigenvalue problem \eqref{Laplacian} in L-shape domain $(-\pi,\pi)^{2}\backslash[0,\pi)\times(-\pi,0]$ and use the triangle $P_{1}$-conforming finite element to compute the first $s$ eigenpairs. First, we choose an initial uniform partition $\mathcal{J}_{H}$ in $\Omega$ with $N=384$, $H=\frac{\sqrt{2}\pi}{2^{3}}$. We refine uniformly the grid layer by layer and fix the ratio $\frac{\delta}{H}=\frac{1}{4}$. Next, we also test the optimality and scalability of our algorithm.
\end{example}

\begin{table}[H]
 \centering
 \caption{ $N=384$, $\frac{\delta}{H}=\frac{1}{4}$, $s=20$}\label{Table3}
\newcolumntype{d}{D{.}{.}{2}}
\begin{tabular}{|c|c|c|c|c|c|c|}
\hline
\multicolumn{1}{|c|}{$d.o.f.$}&\multicolumn{1}{c|}{$12033$}&\multicolumn{1}{c|}{$48641$}
&\multicolumn{1}{c|}{$195585$}&\multicolumn{1}{c|}{$784385$}&\multicolumn{1}{c|}{$3141633$}&\multicolumn{1}{c|}{$12574721$}\\
\hline
\multicolumn{1}{|c|}{$\lambda_{i}$}
&\multicolumn{1}{c|}{$20(it.)$}&\multicolumn{1}{c|}{$21(it.)$}&\multicolumn{1}{c|}{$21(it.)$}
&\multicolumn{1}{c|}{$21(it.)$}&\multicolumn{1}{c|}{$21(it.)$}&\multicolumn{1}{c|}{$21(it.)$}\\
\hline
$\lambda_{1}$ &0.97779160 &0.97710672  &0.97685853 &0.97676592 &0.97673064 &0.97671700\\  \hline
$\lambda_{2}$ &1.54049997 &1.53997797  &1.53984721 &1.53981447 &1.53980628 &1.53980424\\  \hline
$\lambda_{3}$ &2.00120483 &2.00030120  &2.00007530 &2.00001882 &2.00000471 &2.00000118\\  \hline
$\lambda_{4}$ &2.99379382 &2.99181213  &2.99131661 &2.99119271 &2.99116174 &2.99115399\\  \hline
$\lambda_{5}$ &3.23787761 &3.23485049  &3.23390594 &3.23359515 &3.23348782 &3.23344923\\  \hline
$\lambda_{6}$ &4.20803816 &4.20392863  &4.20276004 &4.20241178 &4.20230243 &4.20226625\\  \hline
$\lambda_{7}$ &4.55973910 &4.55561148  &4.55457859 &4.55432018 &4.55425554 &4.55423937\\  \hline
$\lambda_{8}$ &5.00614392 &5.00153601  &5.00038400 &5.00009600 &5.00002400 &5.00000600\\  \hline
$\lambda_{9}$ &5.00710838 &5.00177713  &5.00044429 &5.00011107 &5.00002777 &5.00000694\\  \hline
$\lambda_{10}$&5.75583497 &5.74863270  &5.74667559 &5.74612373 &5.74596090 &5.74591032\\  \hline
$\lambda_{11}$&6.63909768 &6.62779594  &6.62497021 &6.62426358 &6.62408689 &6.62404271\\  \hline
$\lambda_{12}$&7.21353191 &7.20343396  &7.20072285 &7.19997123 &7.19975402 &7.19968809\\  \hline
$\lambda_{13}$&7.26354608 &7.25475836  &7.25256109 &7.25201174 &7.25187440 &7.25184007\\  \hline
$\lambda_{14}$&8.01928106 &8.00481942  &8.00120480 &8.00030120 &8.00007530 &8.00001882\\  \hline
$\lambda_{15}$&9.07298831 &9.05528703  &9.05030402 &9.04883566 &9.04838015 &9.04823118\\  \hline
$\lambda_{16}$&9.37765090 &9.35890171  &9.35420967 &9.35303625 &9.35274286 &9.35266950\\  \hline
$\lambda_{17}$&9.89039396 &9.87265655  &9.86821261 &9.86710081 &9.86682278 &9.86675326\\  \hline
$\lambda_{18}$&10.02364892&10.00592070 &10.00148071&10.00037021&10.00009255&10.00002314\\ \hline
$\lambda_{19}$&10.02376707&10.00592805 &10.00148116&10.00037024&10.00009256&10.00002314\\ \hline
$\lambda_{20}$&10.32191660&10.30197607 &10.29674258&10.29533080&10.29493647&10.29482144\\ \hline
$stop.$       &7.2318e-11 &3.8151e-11  &4.1099e-11 &4.2672e-11 &4.4586e-11 &4.6637e-11\\  \hline
\end{tabular}
\end{table}

\begin{table}[H]
 \centering
 \caption{$N=384,1536,6144$, $\frac{\delta}{H}=\frac{1}{4}$, $s=20$}\label{Table4}
\newcolumntype{d}{D{.}{.}{2}}
\begin{tabular}{|c|c|c|}
\hline
$N$&\multicolumn{1}{c|}{$d.o.f.$}&\multicolumn{1}{c|}{$it.$}\\
\hline
384 &12574721 &21\\   \hline
1536&12574721 &18\\   \hline
6144&12574721 &16\\   \hline
\end{tabular}
\end{table}
\par It is known that some eigenfunctions of the Laplacian eigenvalue problem have singularities at the re-entrant corner but our algorithm still works very well. The number of iterations of our algorithm keeps stable in Table \ref{Table3} as $d.o.f.\to +\infty$, i.e., the convergence rate of our algorithm is independent of $h$. The number of iterations decreases, as the number of subdomains increases in Table \ref{Table4}, which verifies that our algorithm is scalable.

\subsection{3D Laplacian eigenvalue problem}
\par  In order to illustrate that our theoretical analysis still holds for 3D cases, we design two experiments to verify it.
\begin{example}
We consider the Laplacian eigenvalue problem \eqref{Laplacian} in $(0,\pi)^{3}$ and use the trilinear conforming finite element to compute the first $s$ eigenpairs. First, we choose an initial uniform partition $\mathcal{J}_{H}$ in $\Omega$ with $N=512$, $H=\frac{\pi}{2^{3}}$. We refine uniformly the grid layer by layer and fix the ratio $\frac{\delta}{H}=\frac{1}{2}$. Next, we test the optimality and scalability of our algorithm.
\end{example}
\begin{table}[H]
 \centering
 \caption{ $N=512$, $\frac{\delta}{H}=\frac{1}{2}$, $s=20$}\label{Table5}
\newcolumntype{d}{D{.}{.}{2}}
\begin{tabular}{|c|c|c|c|c|c|}
\hline
\multicolumn{1}{|c|}{$d.o.f.$}&\multicolumn{1}{c|}{$3375$}&\multicolumn{1}{c|}{$29791$}
&\multicolumn{1}{c|}{$250047$}&\multicolumn{1}{c|}{$2048383$}&\multicolumn{1}{c|}{$16581375$}\\
\hline
\multicolumn{1}{|c|}{$\lambda_{i}$}
&\multicolumn{1}{c|}{$14(it.)$}&\multicolumn{1}{c|}{$16(it.)$}&\multicolumn{1}{c|}{$16(it.)$}
&\multicolumn{1}{c|}{$17(it.)$}&\multicolumn{1}{c|}{$17(it.)$}\\
\hline
$\lambda_{1}=3$    &3.00965062  &3.00241034  &3.00060244  &3.00015060 &3.00003765   \\  \hline
$\lambda_{2}=6$    &6.05809793  &6.01447439  &6.00361542  &6.00090366 &6.00022590  \\  \hline
$\lambda_{3}=6$    &6.05809793  &6.01447439  &6.00361542  &6.00090366 &6.00022590   \\  \hline
$\lambda_{4}=6$    &6.05809793  &6.01447439  &6.00361542  &6.00090366 &6.00022590   \\  \hline
$\lambda_{5}=9 $   &9.10654523  &9.02653844  &9.00662840  &9.00165671 &9.00041415  \\ \hline
$\lambda_{6}=9 $   &9.10654523  &9.02653844  &9.00662840  &9.00165671 &9.00041415  \\ \hline
$\lambda_{7}=9 $   &9.10654523  &9.02653844  &9.00662840  &9.00165671 &9.00041415  \\ \hline
$\lambda_{8}=11$   &11.26956430 &11.06685176 &11.01667797 &11.00416729&11.00104168  \\ \hline
$\lambda_{9}=11$   &11.26956430 &11.06685176 &11.01667797 &11.00416729&11.00104168  \\ \hline
$\lambda_{10}=11$  &11.26956430 &11.06685176 &11.01667797 &11.00416729&11.00104168  \\ \hline
$\lambda_{11}=12$  &12.15499254 &12.03860249 &12.00964138 &12.00240976&12.00060240  \\ \hline
$\lambda_{12}=14$  &14.31801161 &14.07891581 &14.01969094 &14.00492034&14.00122994  \\ \hline
$\lambda_{13}=14$  &14.31801161 &14.07891581 &14.01969094 &14.00492034&14.00122994  \\ \hline
$\lambda_{14}=14$  &14.31801161 &14.07891581 &14.01969094 &14.00492034&14.00122994  \\ \hline
$\lambda_{15}=14$  &14.31801161 &14.07891581 &14.01969094 &14.00492034&14.00122994  \\ \hline
$\lambda_{16}=14$  &14.31801161 &14.07891581 &14.01969094 &14.00492034&14.00122994  \\ \hline
$\lambda_{17}=14$  &14.31801161 &14.07891581 &14.01969094 &14.00492034&14.00122994  \\ \hline
$\lambda_{18}=17$  &17.36645892 &17.09097986 &17.02270392 &17.00567340&17.00141819  \\ \hline
$\lambda_{19}=17$  &17.36645892 &17.09097986 &17.02270392 &17.00567340&17.00141819  \\ \hline
$\lambda_{20}=17$  &17.36645892 &17.09097986 &17.02270392 &17.00567340&17.00141819  \\ \hline
$stop.$            &8.9527e-11  &3.6276e-11  &5.3570e-11  &2.8785e-11 & 4.5785e-11   \\ \hline
\end{tabular}
\end{table}

\begin{table}[H]
 \centering
 \caption{$N=512,4096$, $\frac{\delta}{H}=\frac{1}{2}$, $s=20$}\label{Table6}
\newcolumntype{d}{D{.}{.}{2}}
\begin{tabular}{|c|c|c|}
\hline
$N$&\multicolumn{1}{c|}{$d.o.f.$}&\multicolumn{1}{c|}{$it.$}\\
\hline
512&16581375&17\\
\hline
4096&16581375&15\\
\hline
\end{tabular}
\end{table}
\par It is seen from Table \ref{Table5} that the number of iterations of our algorithm keeps stable nearly, as $d.o.f.\to +\infty$, which shows that our algorithm is optimal. To verify the scalability of the method, we set $d.o.f.=16581375$ and observe the number of iterations for $N =512,\ 4096$. Numerical results in Table \ref{Table6} show the number of iterations decreases as $N$ increases, which means that the proposed method has a good scalability. Next, we also present some numerical results for three dimensional L-shape domain.

\begin{example}
We consider the Laplacian eigenvalue problem \eqref{Laplacian} in $(0,2\pi)\times (0,2\pi)\times (0,\pi)\backslash
 [\pi,2\pi)\times [\pi,2\pi)\times(0,\pi)$ and use the trilinear conforming finite element to compute the first $s$ eigenpairs. First, we choose an initial uniform partition $\mathcal{J}_{H}$ in $\Omega$ with $N=1536$, $H=\frac{\pi}{2^{3}}$. We refine uniformly the grid layer by layer and fix the ratio $\frac{\delta}{H}=\frac{1}{2}$. Next, we also test the optimality and scalability of our algorithm.
\end{example}

\par It is observed from Table \ref{Table7} that the number of iterations keeps stable nearly when $d.o.f. \to +\infty$, which verifies that the method is optimal for nonconvex domain. In addition, if we observe Table  \ref{Table7} carefully, we may find that some of eigenvalues are close to each other ($\lambda_{17}\approx \lambda_{18}\approx\lambda_{19}$) and our algorithm works still very well, which illustrates that the convergence rate in our two-level BPJD method is not adversely affected by the gap among the clustered eigenvalues. It is obvious to see that the number of iterations decreases as the number of subdomains increases in Table  \ref{Table8}, which shows that our algorithm is scalable.

\begin{table}[H]
 \centering
 \caption{ $N=1536$, $\frac{\delta}{H}=\frac{1}{2}$, $s=20$}\label{Table7}
\newcolumntype{d}{D{.}{.}{2}}
\begin{tabular}{|c|c|c|c|c|}
\hline
\multicolumn{1}{|c|}{$d.o.f.$}&\multicolumn{1}{c|}{$10575$}&\multicolumn{1}{c|}{$91295$}
&\multicolumn{1}{c|}{$758079$}&\multicolumn{1}{c|}{$6177407$}\\
\hline
\multicolumn{1}{|c|}{$\lambda_{i}$}&\multicolumn{1}{c|}{$15(it.)$}&\multicolumn{1}{c|}{$17(it.)$}
&\multicolumn{1}{c|}{$18(it.)$}&\multicolumn{1}{c|}{$18(it.)$}\\
\hline
$\lambda_{1}  $  &1.98468171 &1.97908729  &1.97745736  &1.97695688  \\  \hline
$\lambda_{2}  $  &2.54796654 &2.54184555  &2.54031441  &2.53993134  \\  \hline
$\lambda_{3}  $  &3.00965062 &3.00241034  &3.00060244  &3.00015060  \\  \hline
$\lambda_{4}  $  &4.01258997 &3.99650322  &3.99248898  &3.99148580  \\  \hline
$\lambda_{5}  $  &4.26520929 &4.24231072  &4.23602425  &4.23422553  \\ \hline
$\lambda_{6}  $  &5.03312902 &4.99115134  &4.98047034  &4.97770994  \\ \hline
$\lambda_{7}  $  &5.25609401 &5.21633525  &5.20604572  &5.20330844  \\ \hline
$\lambda_{8}  $  &5.59641384 &5.55390960  &5.54332739  &5.54068439  \\ \hline
$\lambda_{9}  $  &5.61240213 &5.56874371  &5.55786078  &5.55514087  \\ \hline
$\lambda_{10} $  &6.05809793 &6.01447439  &6.00361542  &6.00090366  \\ \hline
$\lambda_{11} $  &6.05809793 &6.01447439  &6.00361542  &6.00090366  \\ \hline
$\lambda_{12} $  &6.05809793 &6.01447439  &6.00361542  &6.00090366  \\ \hline
$\lambda_{13} $  &6.81375190 &6.76357620  &6.75061942  &6.74719395  \\ \hline
$\lambda_{14} $  &7.06103728 &7.00856727  &6.99550196  &6.99223885  \\ \hline
$\lambda_{15} $  &7.31365660 &7.25437477  &7.23903723  &7.23497859  \\ \hline
$\lambda_{16} $  &7.71618710 &7.64709074  &7.62979692  &7.62547071  \\ \hline
$\lambda_{17} $  &8.30454131 &8.22839930  &8.20905870  &8.20215738  \\ \hline
$\lambda_{18} $  &8.34143977 &8.23586694  &8.20907121  &8.20406149  \\ \hline
$\lambda_{19} $  &8.37680647 &8.28281881  &8.25955996  &8.25376045  \\ \hline
$\lambda_{20} $  &8.66084944 &8.58080775  &8.56087375  &8.55589393  \\ \hline
$stop.$          &7.5267e-11 &5.4056e-11  &1.8214e-11  &3.6381e-11  \\ \hline
\end{tabular}
\end{table}

\begin{table}[H]
 \centering
 \caption{$N=1536,12288$, $\frac{\delta}{H}=\frac{1}{2}$, $s=20$}\label{Table8}
\newcolumntype{d}{D{.}{.}{2}}
\begin{tabular}{|c|c|c|}
\hline
$N$&\multicolumn{1}{c|}{$d.o.f.$}&\multicolumn{1}{c|}{$it.$}\\
\hline
1536&6177407&18\\
\hline
12288&6177407&15\\
\hline
\end{tabular}
\end{table}

\section{Conclusions}

\par In this paper, based on a domain decomposition method, we propose a parallel two-level BPJD method for computing multiple and clustered eigenvalues. The method is proved to be optimal, scalable and cluster robust. Numerical results verify our theoretical findings.

\begin{appendices}
\section{\ }
\par\noindent{\bf Proof of Theorem \ref{TheoremPriorLaplacian}}:\ \  Let $ \widetilde{i}$ be an imaginary unit, and $\Gamma$ be a circle which includes $\mu_{1},\mu_{2},...,\mu_{s}$ and $\mu_{1}^{h},\mu_{2}^{h},...,\mu_{s}^{h},$  with $(\frac{\mu_{1}+\mu_{s}}{2},0)$ as a center and $\frac{\mu_{1}-\mu_{s+1}}{2}$ as a radius in complex plane $\mathcal{C}$. Define
\begin{equation}\notag
Z:=\frac{1}{2\pi \widetilde{i}}\int_{\Gamma}(z-T)^{-1}dz,\ \ Z^{h}:=\frac{1}{2\pi \widetilde{i}}\int_{\Gamma}(z-T^{h})^{-1}dz.
\end{equation}
Therefore, we know that $Z$ and $Z^{h}$ are spectral projectors associated with $T$ and $\mu_{1},...,\mu_{s}$, and $T^{h}$ and $\mu_{1}^{h},...,\mu_{s}^{h}$, respectively.
\begin{figure}[H]
  \centering
  \includegraphics[width=11cm,height=6cm]{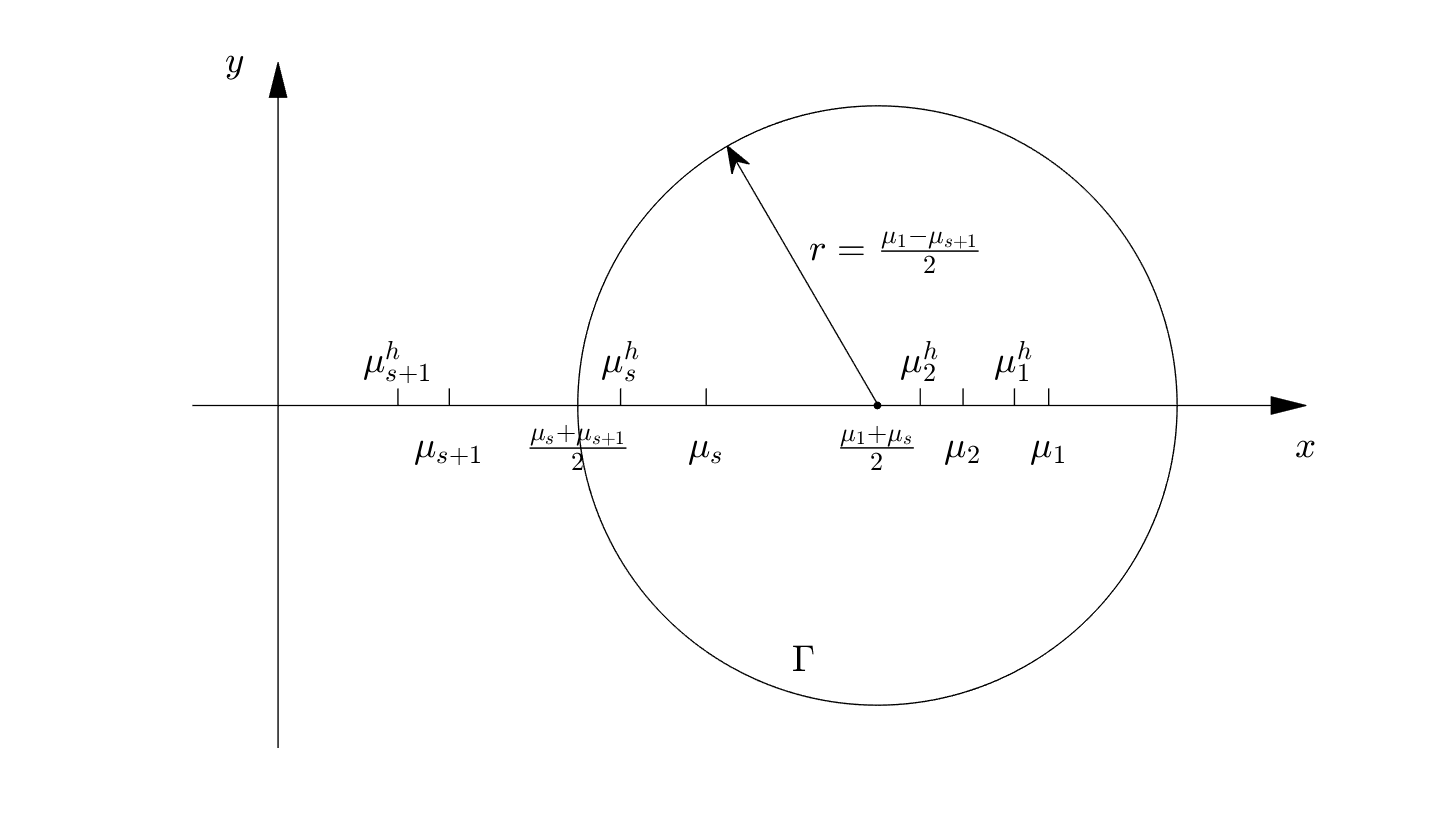}\\
  \caption{$\Gamma$ is a circle which includes $\mu_{1},\mu_{2},...,\mu_{s}$ and $\mu_{1}^{h},\mu_{2}^{h},...,\mu_{s}^{h}.$}\label{Figure1}
\end{figure}
Combining \eqref{equation_DefineT}, \eqref{equation_DefineTh} and
standard a priori error estimates of conforming finite element methods, we have
$$ \interleave  T-T^{h}  \interleave_{a}\leq Ch.$$
Moreover, we have
\begin{equation}\label{Equation_z_Th}
\sup_{z\in \Gamma}\interleave (z-T^{h})^{-1}\interleave_{a}\leq C.
\end{equation}
Since $TT^{*}=T^{*}T$, we get
\begin{equation}\label{Equation_z_T_normal}
(z-T)(\bar{z}-T^{*})=|z|^{2}-zT^{*}-\bar{z}T+TT^{*}
=|z|^{2}-zT^{*}-\bar{z}T+T^{*}T=(\bar{z}-T^{*})(z-T).
\end{equation}
From a priori error estimates in \cite{babuvska1989finite}, we know that $\dim(U_{s})=s=\dim(U_{s}^{h})$,
which, together with \eqref{Equation_z_Th}, \eqref{Equation_z_T_normal} and Remark \ref{Remark_Hausdorff}, yields
\begin{align}
\theta_{a}(U_{s},U_{s}^{h})
&=\sin_{a}\{U_{s},U_{s}^{h}\}=\sup_{u\in U_{s},||u||_{a}=1}\inf_{v\in U_{s}^{h}}||u-v||_{a}
\leq \sup_{u\in U_{s},||u||_{a}=1}||u-Z^{h}u||_{a}\notag\\
&\leq \interleave (Z-Z^{h})|_{U_{s}}\interleave_{a}
=\frac{1}{2\pi}\interleave \int_{\Gamma}(z-T^{h})^{-1}(T-T^{h})(z-T)^{-1}|_{U_{s}}dz\interleave_{a}\notag\\
&\leq \frac{1}{2\pi}|\Gamma|\sup_{z\in \Gamma}\interleave (z-T^{h})^{-1}\interleave_{a}\sup_{z\in \Gamma}\interleave(z-T)^{-1}|_{U_{s}}\interleave_{a} \interleave(T-T^{h})|_{U_{s}}\interleave_{a}\notag\\
&\leq C\sup_{z\in \Gamma}\frac{1}{\inf_{1\leq i\leq s}|z-\mu_{i}|}
\times \interleave (T-T^{h})|_{U_{s}}\interleave_{a}
\leq \frac{Ch}{\mu_{s}-\mu_{s+1}}. \notag
\end{align}
Hence, we obtain\eqref{Equation_eigenvectorApriori1}. Combining the Aubin-Nitsche argument, we have
$$\interleave T-T^{h}\interleave_{b}\leq Ch^{2}.$$
By the same argument as in the proof of \eqref{Equation_eigenvectorApriori1}, we may also obtain \eqref{Equation_eigenvectorApriori2}. \qed
\par To give a rigorous proof of Lemma \ref{Lemma_AuxiliaryProblem} in this paper, we first introduce the following lemma
(For the detailed proof, see Lemma 5 in \cite{ovtchinnikov2002convergence} or Lemma 2 of Appendix A in \cite{ovtchinnikov1}).
For any matrix $X$, we denote by $D_{X}$ the diagonal part of $X$, $\bar{D}_{X}=D_{X}-X$. And we denote by ${\rm Tr}(X)$ the trace of the matrix $X$.
\begin{lemma}\label{Lemma_TraceEquality}
Let $A\ (=\Lambda-\widetilde{A})$ be a symmetric and $B\ (=I-\widetilde{B}) $ a symmetric positive definite matrix, where $\Lambda$ is a diagonal matrix and $I$ is the identity matrix. Then
\begin{equation}\notag
{\rm Tr}(B^{-1}A)={\rm Tr}(D_{B}^{-1}D_{A})-{\rm Tr}(A_{1})+{\rm Tr}(A_{2})+{\rm Tr}(A_{3}),
\end{equation}
where
\begin{align*}
&A_{1}=D_{B}^{-1}\widetilde{B} D_{B}^{-1}(\widetilde{A}+D),\ \ \ A_{2}=D_{B}^{-1}D_{\widetilde{B}} D_{B}^{-1}(D_{\widetilde{A}}+D),\ \ \ A_{3}=D_{B}^{-1}\bar{D}_{B}B^{-1}\bar{D}_{B}D_{B}^{-1}A,
\end{align*}
and $D$ is any diagonal matrix.
\end{lemma}
\par\noindent{\bf Proof of Lemma \ref{Lemma_AuxiliaryProblem}}:\ \
We consider the auxiliary eigenvalue problem \eqref{AuxiliaryEigenvalueProblem} resulting in
\begin{equation}\label{equation_GeneralizedEigenvalue}
A\xi_{i}^{k+1}=\hat{\lambda}^{k+1}_{i}B\xi_{i}^{k+1},
\end{equation}
where $A=(a(\widetilde{u}_{j}^{k+1},\widetilde{u}_{i}^{k+1}))_{1\leq i,j\leq s}$, $B=(b(\widetilde{u}_{j}^{k+1},\widetilde{u}_{i}^{k+1}))_{1\leq i,j\leq s}$ and $\xi_{i}^{k+1}$ is the coordinate of $\hat{u}_{i}^{k+1}$ in the basis $\{\widetilde{u}_{j}^{k+1}\}_{j=1}^{s}$. Define $z_{i}^{k}:=Q_{\perp}^{k}(B_{i}^{k})^{-1}r_{i}^{k},\ i=1,2,...,s$. Substituting \eqref{Equation_construction_utilde} into $b(\cdot,\cdot)$ and $a(\cdot,\cdot)$, we have
\begin{align}
b(\widetilde{u}_{j}^{k+1},\widetilde{u}_{i}^{k+1})=\delta_{ij}+
\alpha_{j}^{k}\alpha_{i}^{k}b(z_{j}^{k},z_{i}^{k})
=:\delta_{ij}+(\hat{B})_{ij},\label{B}
\end{align}
and
\begin{align}
a(\widetilde{u}_{j}^{k+1},\widetilde{u}_{i}^{k+1})
&=\lambda_{j}^{k}\delta_{ij}
+\alpha_{i}^{k}a(u_{j}^{k},z_{i}^{k})
+\alpha_{j}^{k}a(z_{j}^{k},u_{i}^{k})
+\alpha_{j}^{k}\alpha_{i}^{k}a(z_{j}^{k},z_{i}^{k})
=:\lambda^{k}_{j}\delta_{ij}+(\hat{A})_{ij}.\label{A}
\end{align}
By \eqref{B}, it is easy to check that $\hat{B}\geq \bm{0}_{s\times s}$ and $B\ (=I+\hat{B}\ \geq I)$ is symmetric and positive definite. Moreover, by Lemma \ref{Lemma_gik}, Lemma \ref{Lemma_Relationu1e2} and Remark \ref{remark_QusrikEstimate}, we obtain
\begin{align}
&\ \ \ \ ||D_{\hat{B}}||_{F}
\leq||\hat{B}||_{F}\leq \sqrt{s}||\hat{B}||_{2}
\leq \sqrt{s} {\rm Tr}(\hat{B})
=\sqrt{s}\sum_{i=1}^{s}(\alpha_{i}^{k})^{2}||z_{i}^{k}||_{b}^{2}
\leq C\sum_{i=1}^{s}||(B_{i}^{k})^{-1}r_{i}^{k}||_{b}^{2}\notag\\
&\leq C\sum_{i=1}^{s}\{C||e_{i,s+1}^{k}||_{E_{i}^{k}}^{2}+CH^{4}||g_{i}^{k}||_{a}^{2}\}
\leq C\sum_{i=1}^{s}(\lambda_{i}^{k}-Rq(Q_{s}^{h}u_{i}^{k}))+CH^{4}\sum_{i=1}^{s}(\lambda_{i}^{k}-\lambda_{i}^{h}),\label{hatB}
\end{align}
where $||\cdot||_{F}$ and $||\cdot||_{2}$ denote the Frobenius norm and $2$-norm of matrix, respectively.
Using the same argument as in \eqref{hatB}, we deduce
\begin{align}
&\ \ \ \ ||\bar{D}_{B}||_{F}^{2}
= ||\bar{D}_{\hat{B}}||_{F}^{2}
\leq \sum_{i,j=1}^{s}(\alpha_{j}^{k}\alpha_{i}^{k})^{2}|b(z_{j}^{k},z_{i}^{k})|^{2}
\leq C\sum_{i,j=1}^{s}||z_{i}^{k}||_{b}^{2}||z_{j}^{k}||_{b}^{2}\notag\\
&= C\sum_{i=1}^{s}||z_{i}^{k}||_{b}^{2}\sum_{j=1}^{s}||z_{j}^{k}||_{b}^{2}
\leq CH^{2}\sum_{i=1}^{s}(\lambda_{i}^{k}-Rq(Q_{s}^{h}u_{i}^{k}))
+CH^{6}\sum_{i=1}^{s}(\lambda_{i}^{k}-\lambda_{i}^{h}).\label{barDB}
\end{align}
By \eqref{A}, it is easy to check that $A\ (=\Lambda+\hat{A})$ is a symmetric matrix, where $\Lambda=Diag(\lambda_{1}^{k},\lambda_{2}^{k},...,\lambda_{s}^{k})$.
Moreover, by Lemma \ref{Lemma_NormQs+1h}, Lemma \ref{Lemma_Relationu1e2} and Remark \ref{remark_QusrikEstimate}, we get
\begin{align*}
&\ \ \ \ ||D_{\hat{A}}||^{2}_{F}
\leq ||\hat{A}||^{2}_{F}
=\sum_{i,j=1}^{s}\{\alpha_{i}^{k}a(u_{j}^{k},z_{i}^{k})
+\alpha_{j}^{k}a(z_{j}^{k},u_{i}^{k})+\alpha_{j}^{k}\alpha_{i}^{k}a(z_{j}^{k},z_{i}^{k})\}^{2}\\
&\leq C(\sum_{i=1}^{s}||z_{i}^{k}||^{2}_{a}+\sum_{i,j=1}^{s}||z_{j}^{k}||^{2}_{a}||z_{i}^{k}||_{a}^{2})
\leq C(1+\sum_{j=1}^{s}||z_{j}^{k}||^{2}_{a})\sum_{i=1}^{s}||z_{i}^{k}||^{2}_{a}
\leq CH^{2},
\end{align*}
and
\begin{align}
||A||_{2} &\leq {\rm Tr}(A)=\sum_{i=1}^{s}\{\lambda_{i}^{k}+2\alpha_{i}^{k}a(u_{i}^{k},z_{i}^{k})
+(\alpha_{i}^{k})^{2}a(z_{i}^{k},z_{i}^{k})\}\notag\\
&\leq \sum_{i=1}^{s}\{\lambda_{i}^{k}+C||z_{i}^{k}||_{a}
+C||z_{i}^{k}||_{a}^{2}\}
\leq \sum_{i=1}^{s}\lambda_{i}^{k}+CH+CH^{2}\leq C.
\end{align}
By \eqref{AuxiliaryEigenvalueProblem} and \eqref{equation_GeneralizedEigenvalue}, we deduce
\begin{equation}\label{Transformation1}
\sum_{i=1}^{s}(\hat{\lambda}_{i}^{k+1}-Rq(\widetilde{u}_{i}^{k+1}))={\rm Tr}(B^{-1}A)-{\rm Tr}(D_{B}^{-1}D_{A}).
\end{equation}
Using Lemma \ref{Lemma_TraceEquality}, we have
\begin{equation}\label{Transformation2}
{\rm Tr}(B^{-1}A)-{\rm Tr}(D_{B}^{-1}D_{A})=-{\rm Tr}(A_{1})+{\rm Tr}(A_{2})+{\rm Tr}(A_{3})\leq |{\rm Tr}(A_{1})|+|{\rm Tr}(A_{2})|+|{\rm Tr}(A_{3})|,
\end{equation}
where
\begin{align*}
A_{1}&=D_{B}^{-1}(-\hat{B})D_{B}^{-1}(-\hat{A})=D_{B}^{-1}\hat{B}D_{B}^{-1}\hat{A},\\
A_{2}&=D_{B}^{-1}D_{-\hat{B}} D_{B}^{-1}D_{-\hat{A}}=D_{B}^{-1}D_{\hat{B}} D_{B}^{-1}D_{\hat{A}},\\
A_{3}&=D_{B}^{-1}\bar{D}_{B}B^{-1}\bar{D}_{B}D_{B}^{-1}A.
\end{align*}
We first estimate the term $|{\rm Tr}(A_{1})|$ in \eqref{Transformation2}. Since $D_{B}\ (\geq I),\ D_{B}^{-1}\hat{B}D_{B}^{-1}$, $\hat{A}$ and $\hat{B}$ are symmetric, we get
\begin{align}
|{\rm Tr}(A_{1})|
&=|{\rm Tr}(D_{B}^{-1}\hat{B}D_{B}^{-1}\hat{A})|
\leq ||D_{B}^{-1}\hat{B}D_{B}^{-1}||_{F}||\hat{A}||_{F}
\leq ||\hat{B}||_{F}||\hat{A}||_{F}\notag\\
&\leq CH\sum_{i=1}^{s}(\lambda_{i}^{k}-Rq(Q_{s}^{h}u_{i}^{k}))
+CH^{5}\sum_{i=1}^{s}(\lambda_{i}^{k}-\lambda_{i}^{h}).\label{TraceA1}
\end{align}
Next, we estimate the term $|{\rm Tr}(A_{2})|$ in \eqref{Transformation2}. As $D_{B}\ (\geq I),\ D_{B}^{-1}D_{\hat{B}}D_{B}^{-1},\ D_{\hat{A}}$ and $D_{\hat{B}}$ are symmetric, we obtain
\begin{align}
|{\rm Tr}(A_{2})|
&=|{\rm Tr}(D_{B}^{-1}D_{\hat{B}}D_{B}^{-1}D_{\hat{A}})|
\leq||D_{B}^{-1}D_{\hat{B}} D_{B}^{-1}||_{F}||D_{\hat{A}}||_{F}
\leq||D_{\hat{B}}||_{F}||D_{\hat{A}}||_{F}\notag\\
&\leq CH\sum_{i=1}^{s}(\lambda_{i}^{k}-Rq(Q_{s}^{h}u_{i}^{k}))
+CH^{5}\sum_{i=1}^{s}(\lambda_{i}^{k}-\lambda_{i}^{h}).\label{TraceA2}
\end{align}
Finally, we estimate the term $|{\rm Tr}(A_{3})|$ in \eqref{Transformation2}. By the same argument as in \eqref{TraceA2}, we have
\begin{align*}
|{\rm Tr}(A_{3})|
&=|{\rm Tr}(D_{B}^{-1}\bar{D}_{B}B^{-1}\bar{D}_{B}D_{B}^{-1}A)|
\leq ||D_{B}^{-1}\bar{D}_{B}B^{-1}\bar{D}_{B}D_{B}^{-1}||_{F}||A||_{F}\\
&\leq ||\bar{D}_{B}B^{-1}\bar{D}_{B}||_{F}||A||_{F}
\leq  ||\bar{D}_{B}||^{2}_{F}||B^{-1}||_{2}||A||_{F}
\leq ||\bar{D}_{B}||^{2}_{F}||A||_{F}\\
&\leq \sqrt{s}||\bar{D}_{B}||^{2}_{F}||A||_{2}
\leq CH^{2}\sum_{i=1}^{s}(\lambda_{i}^{k}-Rq(Q_{s}^{h}u_{i}^{k}))
+CH^{6}\sum_{i=1}^{s}(\lambda_{i}^{k}-\lambda_{i}^{h}),
\end{align*}
which, together with \eqref{Transformation1}, \eqref{Transformation2},\eqref{TraceA1} and \eqref{TraceA2}, completes the proof of this lemma. \qed
\end{appendices}

\begin{small}
\bibliographystyle{plain}
\bibliography{reference}
\end{small}
\end{document}